\documentclass[11pt]{article}
\usepackage{amsfonts}
\usepackage{mathrsfs}
\usepackage{amsthm}
\usepackage{graphicx}
\usepackage{multirow}
\usepackage{color}
\usepackage{array}
\usepackage{amsmath}
\usepackage{relsize}
\usepackage{cite}
\usepackage{threeparttable}
\usepackage[capposition=bottom,captionskip=1pt]{floatrow}
\usepackage{setspace}
 \usepackage{diagbox}

\newtheorem{theorem}{Theorem}

\usepackage{subcaption} 
\usepackage{booktabs}

\newcommand{\g}{\lambda}               

\newtheorem{lem}{Lemma}
\newtheorem{pro}{Proposition}
\newtheorem{cor}{Corollary}

\def\bsk{\vskip .5cm\noindent}

\def\hsk{\vskip .2cm\noindent}

\begin{document}
\baselineskip 18pt

\title{
{\Large\bf 
Optimal Control of a Non-Stationary Inventory-Queue,\\
with Applications to EV-Battery Swapping}
}
\author{
Hanyu Cheng\thanks{Department of Industrial Engineering,
Tsinghua University.
hanyucheng2026@gmail.com.}
\; \, and \,
David D.\ Yao\thanks{Department of Industrial Engineering and Operations Research,
Columbia University. yao@columbia.edu.}
}

\vskip 1.5cm
\date{September 2026}

\maketitle

\vskip 2cm

\begin{abstract}  

\vskip 0.8cm

We study a non-stationary queueing system with multiple parallel servers and a finite buffer capacity, over a finite planning horizon.  
Motivated by an operation that is central to an EV (electric vehicle) battery-swapping station --- how to best manage
the charging of batteries swapped off from EVs, we develop a continuous-time MDP (Markov decision process) approach to the problem.
The main technical challenge is to allow both the arrival rate and the service-cost rate to be functions of 
time with jumps at discrete time points, key features of the EV swapping station's operation.     
We derive an exact and explicit threshold structure  that characterizes the optimal control policy.
Moreover, we give conditions under which the thresholds are monotone over time, and with respect to key system parameters such as
service and buffer capacities.   
Extensive numerical experiments are carried out to supplement as well as validate the theoretical results, and also to illustrate the robustness 
of the threshold structure in practical implementations.  

\end{abstract}




\thispagestyle{empty}

\newpage

\baselineskip 18pt

\section{Introduction}

We study a {\it non-stationary} queueing system over a {\it finite} time horizon, $t\in [0,T]$, with $K\ge 1$ 
parallel servers and a buffer capacity of $N\ge K$. 
Jobs requiring service arrive at the system following a non-stationary Poisson process with rate $\g (t)$. Any job that arrives at a time when the system is full (i.e., occupied by $N$ jobs, including both in service and in waiting) will be blocked and lost, with a penalty cost $\pi>0$. 
Service times follow an exponential distribution, i.i.d.\ among the servers; and the service-cost rate $e(t)$ --- per time-unit cost charged to each server's 
time while actively serving --- is also time-dependent.  

Our goal is to find an optimal policy that specifies the number of active servers at any time $t \in [0,T]$, so as to minimize the
expected total cost over the entire horizon $[0,T]$ --- service cost plus penalty cost for any lost demand. 

In terms of applications, we are primarily motivated by optimizing an operation that is central to an EV (electric vehicle) battery-swapping station:
The station keeps a fixed number of $N$ batteries, some of which, at any time $t$, can be depleted (e.g., swapped off from an EV, and waiting to be charged); others can be fully charged (and hence ready to supply the demand of any arriving EV that requires battery swapping), or in the process of being charged. The station is equipped with $K$ chargers (``servers''); hence, can charge up to $K$ batteries at any time. Viewed this way, the station is maintaining an inventory of up to $N$ fully-charged batteries; and the process of charging one or more batteries is effectively the process of inventory replenishment. On the other hand, if/when all $N$ batteries are either depleted or in the process of being charged (and hence, are not ready to supply demand), any arriving EV will leave the station, resulting in a lost demand. This corresponds to the case of ``blocking'' (due to a full buffer) in the queueing context. 

Since the demand arrival rates $\g(t)$ and the charging-cost rate $e(t)$ (the cost rate of electricity) are time-dependent --- both are key features of the EV battery-swapping station's operation, it becomes highly non-trivial what the optimal charging (or, inventory-replenishment) strategy is. 
To find the optimal strategy as stated above, we will pursue a stochastic control approach, more specifically, a continuous-time MDP (Markov decision process). 
     
Our study is also motivated by the work of Yang, Yao and Ye (2020) \cite{yangyaoye}, where the focus is on a broad class of so-called ``inventory-queues'' for which the optimal control, for a {\it long-run average} objective, has a threshold structure. (An inventory-queue is a system where the server(s) can still be active even when the system is empty, so as to produce finished products to supply future demand. Thus, the system operates in a mixed make-to-order and make-to-stock modes.) The paper shows that this threshold structure not only applies to an inventory-queue when it operates as a continuous-time Markov chain, but also to the diffusion limit of the original queue when it is non-Markovian (e.g., with renewal arrivals and non-exponential service times). One of our objectives here is to explore similar threshold structures, in a non-stationary setting along with a finite time horizon, as motivated by the EV application stated above. With the results highlighted below, we believe we have successfully completed this first step.       
     
\subsubsection*{Related Literature}

The related literature on service control/optimization in queues predominantly concerns stationary models, e.g., Stidham and Weber (1989) \cite{stidham}, and George and Harrison (2001) \cite{harrison},
with a focus on monotone optimal policies; refer to Glasserman and Yao (1994) \cite{glassermanyao} (Chapter 6). The problem becomes much more difficult under non-stationarity and known results are limited. For instance, Fu, Marcus and Wang (2000) \cite{fu} establish monotone staffing policies for a transient multi-period queue, but staffing decisions are made only at period boundaries. Miller (2009) \cite{miller} studies access and service-rate control over a finite horizon with non-stationary arrivals, and provides a dynamic-programming characterization and numerical solution procedure, but does not derive any explicit structural characterization of the state-dependent service policy. 

The main technical challenge is that non-stationarity and the terminal (end-of-horizon) effect couple the evolution of the value function and the optimal policy over time, making it difficult to characterize the optimal policy analytically. 
For non-stationary finite-buffer $M_t/M/c/N$ queues,
a substantial body of research develops approximations for performance evaluation, e.g., Jagerman (1975) \cite{jagerman} and Massey and Whitt (1996) \cite{massey}, as analytical results under non-stationary arrivals are generally difficult to obtain due to the transient evolution of the system, see Davis, Massey and Whitt (1995) \cite{davis}.


This challenge clearly carries over to the EV battery-swapping problem.
For example, Qi, Zhang and Zhang (2023) \cite{qi},  and Ai, Cheng and Qi (2026) \cite{ai} integrate steady-state/stationary queueing theory and inventory analysis with long-term station-location and capacity decisions. Yet, in practice demand (EV traffic) certainly varies, over time of the day and days of the week, and so does charging cost. Thus, operational decisions such as when to charge depleted batteries and what the right number of fully charged batteries (``inventory level'') to keep at the swapping station is clearly require non-stationary models, and often over a finite horizon (e.g., a month or a quarter) which may consist of several similar cycles (e.g., weeks).
 Studies in this regard appear rather rare. A notable exception is a recent work: Li {\it et al} (2026) \cite{li} study the optimal charger-activation policy over a finite time horizon, and derive the optimal policy with a two-threshold structure, similar to our study. The differences are, a) theirs is a discrete-time model ($t=1,2,\dots, T$), with a deterministic single-period charging time; b) they consider a dual source of power for charging: from solar panel (locally installed at the station) and from the power grid, which is a feature not considered in our model.

Notably, our basic model (detailed in \S\ref{sec:optcontrol} below) goes beyond the $M_t/M/c/N$ queues in that we allow the service-cost rate to be time-dependent, as well as the arrival rate. Moreover, we allow these time-dependent rates to have jumps at discrete time points over the horizon, so as to capture key features, such as higher rush-hour traffic and lower electricity cost for off-peak (e.g., overnight) charging, that are critical in the applications that motivate our study. Indeed, these discontinuities are among the main technical challenges in our analyses.   

\subsubsection*{Key Contributions} 

\begin{itemize}

\item Applying continuous-time MDP to study the optimal control of a non-stationary inventory-queue over a finite time horizon, we obtain an exact and explicit threshold structure (Theorem \ref{theo:threshold}) that characterizes the optimal control policy, using a combination of  Dynkin's formula and Nagumo's theorem. 
Specifically, the optimal replenishment strategy, for each time-state pair, has a
bang-bang structure: either replenish all idle/depleted units (up to $K$) or no unit at all. Furthermore, the
optimal control exhibits a (stronger) threshold structure, i.e., activating the replenishment process
if and only if the number of depleted units have reached (or exceeded) a threshold $i^{\min}_t$ (which only
depends on time). 

\item We give both necessary and sufficient conditions (in Theorem \ref{theo:threshold_monotonicity}), based on
comparing the charging-cost rate versus the demand rate, for the threshold $i^{\min}_t$ to be increasing (non-decreasing) in $t$.
In addition, we also show (in Theorem \ref{theo:capacity_control_interaction}) the threshold's monotonicity w.r.t.\ system capacities ($K$ and $N$).

\item For non-exponential charging times, we use the phase-type (PH) distribution, which is
known to be dense in the class of distributions of non-negative valued random variables. Since the
PH distribution can be represented as the absorbing time of a Markov chain, the MDP framework
still applies, and numerical experiments demonstrate that the optimal control still exhibits a threshold structure, thus, supporting the robustness of the structure.

\end{itemize}

The rest of the paper is organized as follows. In \S\ref{sec:optcontrol} we first derive the bang-bang structure of the optimal control and then
strengthen it to a threshold structure (as highlighted above). We then continue to reveal in \S\ref{sec:monotonicity} the monotonicity of thresholds 
both over time and in terms of the capacity parameters, the number of servers $K$ and the buffer capacity $N$. 
\S\ref{sec:cyclic_model} focuses on the case when the planning horizon $[0,T]$ consists of cycles (such as peak-to-off-peak), where we reveal certain cyclic monotonicity of the thresholds.
In \S\ref{sec:extension} we extend the exponential charging-time distribution to the general class of phase-type distributions, for which the threshold structure of the optimal control is difficult to pin down analytically, but readily demonstrated numerically in the extensive numerical experiments detailed in \S\ref{sec:numerical}. Concluding remarks are summarized in \S\ref{sec:conclusion}.

\section{Optimal Replenishment Control}
\label{sec:optcontrol}

\subsection{Model and Problem Formulation}
\label{sec:model}

We study an inventory system of $N$ units, where $N\ge 1$ is a given integer. At any time, each of the $N$ units is either in an ``active/ready'' state or in an ``idle/not ready'' state, according to whether or not it can be used to supply a demand. In the earlier motivating example of an EV battery swapping station, $N$ corresponds to the total number of batteries at the station: some are fully charged, ready to supply any incoming EV that requires a swap; others are being charged or waiting to be charged. Hence, the charging process corresponds to inventory replenishment.

Time is continuous: $t\in [0,T]$, with $T>0$ given and finite; and we shall refer to $[0,T]$ as the planning horizon. Let $I(t)$ denote the number of idle/not ready units, which will be treated as a state variable, with state space
$
\mathcal{S}=\{0,1,\ldots,N\}.
$
Specifically, when $I(t)=i$, there are 
$N-i$ units ready to satisfy demand. To the extent that $N$ is the maximal number of ready units, this corresponds to what is traditionally called a ``base-stock'' inventory replenishment system, with $N$ being the base-stock level. Note that any one of the $I(t)$ idle units may or may not be under replenishment, depending on the replenishment strategy.

Demands arrive at the system following a non-stationary Poisson process with time-varying intensity $\lambda(t)$, where $\lambda(t)$ is deterministic, bounded, right-continuous, $\inf_{t\in[0,T]} \lambda(t)>0$, and continuously differentiable except at finitely many time points. Each arriving demand requires one ready unit. If $I(t)=i<N$, the demand is fulfilled immediately, and the state changes from $i$ to $i+1$; if $I(t)=N$, i.e., no ready unit is available, the demand is lost and the state remains unchanged. In the EV example, these correspond to swapping times being negligible (demand fulfilled immediately), and no EV will wait for swapping if there's no ready battery (lost demand when inventory is empty).

The idle units can be replenished simultaneously (i.e., in parallel), up to $K$ at any time, where $K\ge 1$ is a given integer.
Thus, $K$ can represent the number of battery chargers installed at the swapping station.
In this application context, we shall refer to a charger as a ``server'', and an ``active server'' means a charger that is in the process of charging a battery.
In the more general inventory setting, an active server corresponds to a unit of the depleted inventory being replenished (by a ``server'') -- a replenishment order has been issued and the unit is on its way.

 For each idle unit, the replenishment time follows an independent and identically distributed (iid) exponential distribution with rate $\mu$, which is also independent of the demand process.
The replenishment control is to decide the number of active servers from the following feasible set:
\def\ik{{ i\wedge K}}
\begin{equation}
\label{Uset}
\mathcal{U}(i)=\{0,1,\ldots, \ik\},
\end{equation}
where $\wedge$ denotes the $\min$ operator.
A control (or ``action''),  $u(t)=k$, means activating $k\le K$ servers. Under this action, replenishment completions occur at an aggregate rate $k\mu$, and each completion decreases the state from $i$ to $i-1$. By measuring time in units of the mean replenishment time, we normalize $\mu=1$ without loss of generality. Server activation and deactivation are assumed to be instantaneous and incur no additional (switching) cost. A control $u=\{u(t):t\in[0,T]\}$ is admissible if it is non-anticipative (i.e., independent of any future information) and satisfies $u(t)\in\mathcal{U}(I(t))$ at all times.


Under an admissible control, the state process $\{I(t):t\in[0,T]\}$ is a controlled birth-death process, of which the infinitesimal generator $Q(t,u)=[q_{ij}(t,u)]_{i,j=0}^N$ takes the following form. Its (non-zero) off-diagonal elements of the infinitesimal generator are:
\begin{equation}
\label{eq:transition_rates}
q_{i,i+1}(t,u)=\lambda(t),
\qquad
q_{i,i-1}(t,u)=u ;
\end{equation}
and the diagonal elements are:
\[
q_{ii}(t,u)=-\sum_{j\ne i}q_{ij}(t,u).
\]
(All other entries of $Q(t,u)$ are zero). An arrival in state $N$ leaves the state unchanged (hence not affecting $Q$). As this corresponds to a lost demand, it will incur a penalty cost $\pi>0$.

Let $e (t):[0,T]\to(0,\infty)$ denote the operating cost per active server per time unit (the charging cost rate).
Same as the assumptions on the arrival rate  $\lambda (t)$, we assume the function $e (t)$ is deterministic, bounded, right-continuous, $\inf_{t\in[0,T]} e(t)>0$, and continuously differentiable except at finitely many time points.
Thus, when the system is in state $i$ and action $u$ is chosen, the expected instantaneous cost rate is
\begin{equation}
\label{eq:running_cost}
c(t,i,u)=e(t)u+\pi\lambda(t)\mathbf{1}_{\{i=N\}}.
\end{equation}

For an admissible control $u$, denote the (expected) cost-to-go function by
\begin{equation}
\label{eq:policy_cost}
J^u(t,i)
=\mathbb{E}\left[
\int_t^T c\bigl(s,I(s),u(s,I(s))\bigr)\,ds
\right].
\end{equation}
This definition applies to $(t,i)\in[0,T]\times\mathcal{S}$, where the expectation is with respect to the state $I(s)$.
Moreover, denote the value function as follows, along with the terminal condition:
\begin{equation}
\label{eq:value_function}
\begin{aligned}
V(t,i)&:=\inf_{u}J^u(t,i),
&& (t,i)\in[0,T]\times\mathcal{S},\\
V(T,i)&:=0,
&& i\in\mathcal{S}.
\end{aligned}
\end{equation}

Finding the optimal replenishment strategy is to dynamically adjust the number of active servers -- call it ``activation strategy'' -- based on comparing the (current) operating cost rate $e(t)$ against the (future) cost {\it marginal} (of having one fewer ready unit), with the latter denoted by
\begin{equation}
\label{delta}
\begin{aligned}
\Delta_0(t)&:=V(t,0),\\
\Delta_i(t)&:=V(t,i)-V(t,i-1),
&& i=1,\ldots,N.
\end{aligned}
\end{equation}

Having set up the model and the control problem, in what follows we will first show in \S\ref{sec:bangbang} that the optimal activation strategy, for each time-state pair, has a bang-bang structure: either replenish all idle units (up to $K$) or no unit at all. We then examine in \S\ref{sec:state_threshold} how this decision varies across states and establish a threshold on the number of idle units.

\subsection{Bang--Bang Structure}
\label{sec:bangbang}

\begin{pro}
\label{pro:optimal_control}
{\rm
For the optimal control problem in (\ref{eq:running_cost}), (\ref{eq:policy_cost}) and (\ref{eq:value_function}),
there exists an optimal Markovian policy $u^*$:
\begin{equation}
\label{eq:optimal_control}
u^*(t,i)=
\begin{cases}
\ik , & \text{if }\Delta_i(t)>e(t),\\
0, & \text{if }\Delta_i(t)\le e(t),
\end{cases}
\end{equation}
where $\Delta_i(t)$ follows (\ref{delta}).
\hfill$\Box$
}
\end{pro}

Intuitively, since $\Delta_i(t)$ represents the marginal continuation cost of having one fewer ready unit available to satisfy demand, if $\Delta_i(t)>e(t)$ it certainly makes sense to activate all available servers.

To derive the optimal control in (\ref{eq:optimal_control}) from the HJB (Hamilton-Jacobi-Bellman) equations,
the technical challenge is to handle the discontinuities of both $\lambda (t)$ and $e(t)$, which are essential features of our model. Specifically, let
\begin{equation}
\label{jump}
0=\tau_0<\tau_1<\cdots<\tau_m<\tau_{m+1}=T
\end{equation}
be all the discontinuity points of $\lambda$ and $e$, so that at least one of the two functions will have jumps at these points, even though both are continuously differentiable on each open interval $(\tau_k,\tau_{k+1})$.

The proof of the above proposition in Appendix \ref{app:proof_hjb} will focus on a piecewise solution to the HJB system: $W=(W_0,\ldots,W_N)$, with each component $W_i$ being continuously differentiable on $(\tau_k,\tau_{k+1})$. With the boundary conventions
$W_{N+1}(t)=W_N(t)$ and $W_{-1}(t)=W_0(t)$, the HJB equations can be written uniformly as
\begin{equation}
\label{eq:piecewise_hjb}
\begin{aligned}
-\dot W_i(t)
={}\min_{u\in\mathcal U(i)}
\Bigl\{c(t,i,u)+\lambda(t)\bigl(W_{i+1}(t)-W_i(t)\bigr)+u\bigl(W_{i-1}(t)-W_i(t)\bigr)\Bigr\},
\end{aligned}
\end{equation}

Specifically, $W_i$ will be constructed as follows:
Start from the terminal condition $W_i(T)=0$,
working backward over the intervals: $k=m,\dots, 0$. For the $k$-th interval $(\tau_k,\tau_{k+1})$,
solve the HJB equation in (\ref{eq:piecewise_hjb}), with the boundary condition $W_i(\tau_{k+1}^-)=W_i(\tau_{k+1})$.
Next, repeat the same for the interval $(\tau_{k-1},\tau_{k})$, with the boundary condition $W_i(\tau_{k}^-)=W_i(\tau_{k})$.

What remains is to argue that $W$ so constructed is continuous on $[0,T]$ and $C^1$ on each of the $m+1$ open intervals, and to verify that $W_i(t)=V(t,i)$ for all $t\in[0,T]$ and all $i \in \mathcal S$. The full proof is detailed in Appendix \ref{app:proof_hjb}.

\subsection{Threshold Structure}
\label{sec:state_threshold}

Given the bang-bang structure of the optimal control in each state $i$ in Proposition \ref{pro:optimal_control},
i.e., activate all available servers if $\Delta_i(t) >e(t)$, 
we shall call a state $i$ that satisfies this condition an ``activation state''.
Moreover,  denote
for each $t\in[0,T]$, the set of all activation states:
\begin{equation}
\label{activate}
\mathcal{I}_t:=\{i\in\mathcal{S}\setminus\{0\} :\, \Delta_i(t)>e(t)\}.
\end{equation}

It turns out that the marginal continuation costs $\Delta_i(t):=V(t,i)-V(t,i-1)$, $i\in\mathcal{S}$, are increasing in $i$; or equivalently, $V(t,i)$ is convex in $i$.

\begin{lem} 
\label{lem:monotonicity}
{\rm
The marginal costs $\Delta_i(t) := V(t, i) - V(t, i-1)$ satisfy:
\begin{equation}
\label{delta1}
\pi\ge\Delta_N(t) \ge \Delta_{N-1}(t) \ge \cdots
\ge \Delta_{1}(t)\geq0,
\ \ t \in [0, T].
\end{equation}
Hence, $V(t, i)$ is convex in $i$.
Furthermore, all inequalities in (\ref{delta1}), except the last one $\Delta_1(t)\ge 0$, hold strictly for $t\in[0,T)$.
\hfill$\Box$
}
\end{lem}

Intuitively, as $i$ increases, the system has fewer ready units and is closer to stockout. Losing one more ready unit is more likely to result in lost demand: when many ready units are available, losing one unit has little effect because the system has enough buffer; when only a few ready units remain, losing one more unit is more likely to cause lost demand. Hence, losing one more ready unit has a larger marginal continuation cost as $i$ increases. This marginal cost cannot exceed $\pi$, since one additional ready unit can avoid at most one lost-demand penalty.

To prove the ordering in Lemma \ref{lem:monotonicity}, the technical challenge is that the optimal control enters the HJB system through $\Psi_i(t):=(i\wedge K)\min\{0,e(t)-\Delta_i(t)\}$, making the sign of $\Psi_i(t)-2\Psi_{i-1}(t)+\Psi_{i-2}(t)$ difficult to determine. We analyze it on the boundary faces (e.g., $\Delta_i(t)=\Delta_{i-1}(t)$) and apply Nagumo's theorem. In the proof of the above lemma in Appendix \ref{app:DC}, we reverse time and consider the region
\begin{equation}
\mathcal C:=\left\{(\Delta_1,\ldots,\Delta_N):
0\leq\Delta_1\leq\cdots\leq\Delta_N\leq\pi\right\}.
\end{equation}
What remains to prove is that the backward-time trajectory remains in $\mathcal C$, starting from $\Delta_i(T)=0$. We verify that the backward-time derivatives of the marginal costs do not point outside $\mathcal C$ on any boundary face. For example,
\begin{equation*}
-\frac{d}{dt}\bigl(\Delta_i-\Delta_{i-1}\bigr)\geq 0
\quad\text{when}\quad
\Delta_i-\Delta_{i-1}=0.
\end{equation*}
Nagumo's theorem then implies that $\mathcal C$ is invariant, which yields the ordering in \eqref{delta1}. The full proof is detailed in Appendix~\ref{app:DC}.

This monotonicity leads to the threshold structure of the optimal control in Theorem \ref{theo:threshold}.


\begin{theorem} 
\label{theo:threshold}
{\rm
Denote $i_t^{\min} := \min\mathcal{I}_t$ as the lowest/smallest state among all the activation states.
The optimal control in Proposition \ref{pro:optimal_control} can be strengthened as follows:
for all $t\in (0,T)$ and all $i\in \mathcal{S}$,  $u^*(t,i)=\ik$, if $i\ge i_t^{\min}$; otherwise, $u^*(t,i)=0$.
(If $\mathcal{I}_t=\emptyset$, then $u^*(t,i)=0$ by default.)
}
\end{theorem}

\begin{proof}
Fix $t\in[0,T]$. By Proposition \ref{pro:optimal_control}, the selected optimal action is
\begin{equation*}
u^*(t,i)=
\begin{cases}
i\wedge K, & \Delta_i(t)>e(t),\\
0, & \Delta_i(t)\le e(t).
\end{cases}
\end{equation*}
By Lemma \ref{lem:monotonicity}, $\Delta_i(t)$ is increasing in $i$. Thus, if $k\in\mathcal I_t$, then, for every $j>k$,
\begin{equation}
\Delta_j(t)\ge\Delta_k(t)>e(t),
\end{equation}
so $j\in\mathcal I_t$. Hence, $\mathcal I_t$ is an upper interval of $\{1,\ldots,N\}$. If $\mathcal I_t\neq\emptyset$, then $i_t^{\min}=\min\mathcal I_t$ is well defined. By the definition of $i_t^{\min}$,
\begin{equation}
\begin{aligned}
\Delta_i(t)&\le e(t), && i<i_t^{\min},\\
\Delta_i(t)&>e(t), && i\ge i_t^{\min}.
\end{aligned}
\end{equation}
Proposition \ref{pro:optimal_control} therefore gives
\begin{equation*}
u^*(t,i)=
\begin{cases}
i\wedge K, & i\ge i_t^{\min},\\
0, & i<i_t^{\min}.
\end{cases}
\end{equation*}
If $\mathcal I_t=\emptyset$, then $\Delta_i(t)\le e(t)$ for every $i=1,\ldots,N$, and Proposition \ref{pro:optimal_control} gives $u^*(t,i)=0$ in every state. Under the convention $i_t^{\min}=N+1$, the same threshold representation remains valid. This proves the theorem.
\end{proof}

The intuition behind the above theorem is the combination of Proposition \ref{pro:optimal_control} and Lemma \ref{lem:monotonicity}: if replenishment is worthwhile (i.e., $\Delta_i(t)>e(t)$) at a given state, it remains worthwhile whenever fewer ready units are available, as replenishment is more valuable when the system is closer to stockout (i.e., $\Delta_i(t)>\Delta_{i-1}(t)$).

The above theorem strengthens the earlier bang--bang characterization of the optimal action via $\Delta_i(t)$, which depends on both time and state, to a threshold $i_t^{\min}$ that only depends on time. At any time $t$, the threshold is the {\it lowest} activation state, i.e., the smallest number of idle units that warrant activation. Having been already identified as the lowest (among all activation states), the threshold has removed any state dependency.

\section{Monotonicity of the Threshold}
\label{sec:monotonicity}

\subsection{Temporal Monotonicity}
\label{sec:temporal_switch}

As time advances, there is less time to use replenished units. This motivates us to examine how the threshold $i_t^{\min}$ evolves over time. Specifically, we identify conditions under which $i_t^{\min}$ is increasing in time. Under this temporal monotonicity, the set $\mathcal{I}_t$ will shrink as $t$ increases to  $T$. Consequently, for any given state $i$, the optimal action can only switch from activation to inaction at most once and cannot subsequently be reactivated.


\begin{lem}
\label{lem:lower_bound_D}
{\rm
Denote $\delta_i(t):=\Delta_i(t)-\Delta_{i-1}(t)$, for
$i=2,\ldots,N$, and let
$\Lambda(t,s):=\int_t^s\lambda(r)\,dr$.
For every $i=2,\ldots,N$ and $t\in[0,T]$,
\begin{equation}
\label{deltalb}
\delta_i(t)\ge \underline\delta_i(t)
:=\pi e^{-\Lambda(t,T)-K(T-t)}
\frac{\Lambda(t,T)^{N-i+1}}{(N-i+1)!}.
\end{equation}
Moreover, these bounds satisfy $\underline\delta_i(t)\ge\underline\delta(t)$ where
\begin{equation}
\label{deltalb00}
\begin{aligned}
\underline\delta(t)\!:=\pi e^{-\Lambda(t,T)-K(T-t)}
\!\min\left\{\Lambda(t,T),
\frac{\Lambda(t,T)^{N-1}}{(N-1)!}\right\}\!.
\end{aligned}
\end{equation}
}
\hfill$\Box$
\end{lem}


Observe that $\delta_i(t)$ measures the increase in the marginal value of replenishment as inventory becomes more depleted. Without replenishment, exactly $N-i+1$ arrivals cause one lost demand from state $i$, but none from states $i-1$ and $i-2$. The difference between the two marginal losses is therefore $\pi$. The probability of this event is the Poisson arrival probability
\begin{equation*}
e^{-\Lambda(t,T)}
\frac{\Lambda(t,T)^{N-i+1}}{(N-i+1)!}
\end{equation*}
multiplied by $e^{-K(T-t)}$, the probability of no completion at the maximum replenishment rate $K$. Multiplying this probability by $\pi$ gives the state-dependent lower bound $\underline\delta_i(t)$.

The technical challenge in proving Lemma \ref{lem:lower_bound_D} lies in controlling the second difference of the term $\Psi_i(t):=(i\wedge K)\min\{0,e(t)-\Delta_i(t)\}$ in the HJB equations. The proof will first bound this second difference from below, and then integrate the HJB system with this bound, yielding the lower bounds on $\delta_i(t)$ in \eqref{deltalb}. The full proof is detailed in Appendix \ref{app:time-monotonicity}.

The lower bound $\underline\delta(t)$ leads to a sufficient condition under which the threshold $i_t^{\min}$ is increasing in time, as stated in the theorem below.

\begin{theorem}
\label{theo:threshold_monotonicity}
{\rm Suppose the operating cost has no downward jumps, i.e., $e(\tau_k)\ge e(\tau_k^-)$ at every discontinuity point $\tau_k$ in (\ref{jump}).

\hsk
(i) ({\it Sufficient}) If $\dot e(t)>-\lambda(t)
\min\left\{\underline \delta(t),\pi e^{-\Lambda(t,T)}\right\}$, then $i_t^{\min}$ is increasing in $t$.

\hsk
(ii) ({\it Necessary}) If $i_t^{\min}$ is increasing in $t$, then $\dot e(t)\ge-\lambda(t)\bigl(\pi-e(t)\bigr)$ whenever $\Delta_{i_t^{\min}-1}(t)=e(t)$.
}
\end{theorem}
\begin{proof}
For $i=1,\ldots,N$, at any differentiability point where $\Delta_i(t)=e(t)$, Lemma \ref{lem:monotonicity} implies $\Psi_i(t)=\Psi_{i-1}(t)=0$, with $\Psi_0(t)=0$. The HJB system therefore gives
\begin{equation}
\label{eq:temporal_boundary}
\frac{d}{dt}\bigl(\Delta_i(t)-e(t)\bigr)
=-\dot e(t)-\lambda(t)
\begin{cases}
\delta_{i+1}(t), &\mkern-10mu i<N,\\
\pi-e(t), &\mkern-10mu i=N.
\end{cases}
\end{equation}

We first prove the {sufficient} condition. By Lemma \ref{lem:lower_bound_D} and $\pi-\Delta_N(t)\ge\pi e^{-\Lambda(t,T)}$ (see \eqref{eq:threshold_delta_N_upper_bound} in Appendix \ref{app:DC}), condition (i) makes the derivative in (\ref{eq:temporal_boundary}) strictly negative. More generally, whenever $\Delta_i(t)\ge e(t)$, the HJB system and the same bounds give
\begin{equation}
\label{ieq:diff}
\begin{aligned}
\dot\Delta_i(t)-\dot e(t)
=u_i\bigl(\Delta_i(t)-e(t)\bigr)+\Psi_{i-1}(t)-\lambda(t)\delta_{i+1}(t)-\dot e(t)<u_i\bigl(\Delta_i(t)-e(t)\bigr).
\end{aligned}
\end{equation}
Suppose that $\Delta_i(t)-e(t)$ becomes positive after being nonpositive within a continuity interval. By continuity, there exist $t_0<t_1$ such that $\Delta_i(t_0)-e(t_0)=0$ and $\Delta_i(s)-e(s)>0$ for every $s\in(t_0,t_1]$. Integrating \eqref{ieq:diff} with the factor $e^{-u_i s}$ gives 
\begin{equation}
\begin{aligned}
e^{-u_i t_1}\bigl(\Delta_i(t_1)-e(t_1)\bigr)=\mkern-3mu\int_{t_0}^{t_1}e^{-u_i s} \Bigl[\dot\Delta_i(s)-\dot e(s)-u_i\bigl(\Delta_i(s)-e(s)\bigr)\Bigr]\,ds<0.
\end{aligned}
\end{equation}
This contradicts $\Delta_i(t_1)-e(t_1)>0$. Thus, an inactive state cannot become active within a continuity interval. At a discontinuity point, $\Delta_i(t)$ is continuous and $e(t)$ has no downward jump, so their difference cannot jump upward either. Thus, an inactive state remains inactive, and $i_t^{\min}$ is increasing.

For the {necessity} condition, fix a differentiability point $t$ such that $i:=i_t^{\min}-1\in\{1,\ldots,N\}$ and $\Delta_i(t)=e(t)$. Since the threshold is increasing, state $i$ remains inactive after $t$, implying $\dot\Delta_i(t)-\dot e(t)\le0$. For $i<N$, (\ref{eq:temporal_boundary}) gives $\dot e(t)\ge-\lambda(t)\delta_{i+1}(t)$. Since $\delta_{i+1}(t)=\Delta_{i+1}(t)-e(t)\le\pi-e(t)$, we obtain $\dot e(t)\ge-\lambda(t)(\pi-e(t))$. For $i=N$, the same conclusion follows directly from (\ref{eq:temporal_boundary}).
\end{proof}

The sufficient and necessary conditions differ in how the relevant marginal gap is bounded. The sufficient condition uses the common lower bound in Lemma \ref{lem:lower_bound_D}, which uniformly rules out reactivation at every possible state boundary: the operating cost may decrease, but more slowly than the marginal value of replenishment. The necessary condition instead uses the universal upper bound $\pi-e(t)$ at the state boundary actually reached by the optimal control: the operating cost cannot decrease faster than the marginal value can possibly decline, since doing so would make replenishment worthwhile again.

An increasing operating cost $e(t)$ guarantees temporal monotonicity of the threshold. The following corollary states this consequence on any interval of the planning horizon and will be used in later sections.

\begin{cor}
\label{cor:increasing_operating_cost}
{\rm
Let $[a,b]\subseteq[0,T]$. Suppose that $\dot e(t)\ge0$ at every continuity point $t\in(a,b)$ and that $e(\tau_k)\ge e(\tau_k^-)$ at every discontinuity point $\tau_k\in(a,b]$. Then $i_t^{\min}$ is increasing on $[a,b]$.
}
\end{cor}

\begin{proof}
For every $t<T$, the assumption $\inf_{s\in[0,T]}\lambda(s)>0$ implies $\Lambda(t,T)>0$. Hence, Lemma \ref{lem:lower_bound_D} gives $\underline\delta(t)>0$, while clearly $\pi e^{-\Lambda(t,T)}>0$. Therefore, at every continuity point $t\in(a,b)$,
\begin{equation}
\label{eq:cor_increasing_cost_condition}
\dot e(t)\ge0>
-\lambda(t)\min\left\{\underline\delta(t),
\pi e^{-\Lambda(t,T)}\right\}.
\end{equation}
Thus, the sufficient condition in Theorem \ref{theo:threshold_monotonicity} holds.
This proves that $i_t^{\min}$ is increasing on $[a,b]$.
\end{proof}

\subsection{Monotonicity in Resource Capacity}
\label{sec:capacity_control_interaction}

Inventory provides a buffer against demand, while replenishment capacity determines how many depleted units can be replenished simultaneously. We examine how these two capacities, $N$ and $K$, affect the activation threshold in Theorem \ref{theo:threshold}. To make this dependence explicit, let $V^{N,K}(t,i)$ denote the value function for a system with $N$ units and at most $K$ replenishment servers, and define
\begin{equation}
\Delta_i^{N,K}(t):=V^{N,K}(t,i)-V^{N,K}(t,i-1),
\end{equation}
for $i=1,\ldots,N$. Following \eqref{activate}, denote the corresponding set of activation states by
\begin{equation}
\mathcal I_t^{N,K}:=
\left\{i\in\mathcal{S}\setminus\{0\}:
\Delta_i^{N,K}(t)>e(t)\right\}.
\end{equation}
Let $i_{N,K}^{\min}(t)$ denote the smallest activation state in a system with capacities $N$ and $K$, with the convention that $i_{N,K}^{\min}(t)=N+1$ if $\mathcal I_t^{N,K}=\emptyset$.

\begin{pro}
\label{pro:capacity_marginal_cost}
{\rm
For every $t\in[0,T]$, the following statements hold.
\begin{enumerate}
    \item For every fixed $K$, every $N\ge1$, and every common state $i=1,\ldots,N$, $\Delta_i^{N+1,K}(t)\le\Delta_i^{N,K}(t)$.

    \item For every fixed $N$, every $K\ge1$, and every state $i=1,\ldots,N$, $\Delta_i^{N,K+1}(t)\le\Delta_i^{N,K}(t)$.
    \hfill$\Box$
\end{enumerate}
}
\end{pro}

At a fixed state $i$, increasing $N$ provides more ready units and a larger buffer against demand. Losing one additional ready unit is less likely to cause stockout and has a smaller marginal continuation cost. Increasing $K$ instead gives the system greater capacity to replenish depleted units before stockout occurs. The effect of losing one ready unit can thus be more readily offset by replenishment, reducing its marginal continuation cost.

The main technical challenge is that changing $N$ or $K$ changes the HJB system itself, so the marginal continuation costs of the two systems cannot be compared directly. The proof provided in Appendix \ref{app:capacity_control_interaction} uses Nagumo's theorem to establish the invariance of the region in which the larger-capacity system has no greater marginal continuation costs. On each boundary face, we show that the backward-time derivative points into this region; for example,
\begin{equation*}
-\frac{d}{dt}
\bigl(\Delta_i^{N,K}(t)-\Delta_i^{N+1,K}(t)\bigr)\ge0
\end{equation*}
whenever $\Delta_i^{N,K}(t)=\Delta_i^{N+1,K}(t)$ and $\Delta_j^{N,K}(t)\ge\Delta_j^{N+1,K}(t)$ for every common state $j$. The key distinction between the two comparisons lies in where their HJB systems differ. For the comparison in $N$, the two marginal-cost HJB systems have the same interior structure over their common states, but differ at the upper boundary because the $(N+1)$-unit system has one additional state. For the comparison in $K$, the maximum feasible actions $i\wedge K$ and $i\wedge(K+1)$ coincide for $i\le K$ and differ for $i\ge K+1$. The full proof is detailed in Appendix \ref{app:capacity_control_interaction}.

The above proposition directly leads to the following Theorem \ref{theo:capacity_control_interaction}.

\begin{theorem}
\label{theo:capacity_control_interaction}
{\rm
Denote $i_{N,K}^{\min}(t):=\min\mathcal I_t^{N,K}$ as the lowest/smallest state among all the activation states in a system with capacities $N$ and $K$. If $\mathcal I_t^{N,K}=\emptyset$, set $i_{N,K}^{\min}(t):=N+1$. Then, for every $t\in[0,T]$, the following statements hold.
\begin{enumerate}
    \item For every fixed $K$ and every $N\ge1$, $i_{N+1,K}^{\min}(t)\ge i_{N,K}^{\min}(t)$.

    \item For every fixed $N$ and every $K\ge1$, $i_{N,K+1}^{\min}(t)\ge i_{N,K}^{\min}(t)$.
\end{enumerate}
}
\end{theorem}

\begin{proof}
Fix $t\in[0,T]$. If $i\in\mathcal I_t^{N+1,K}\cap\{1,\ldots,N\}$, then $\Delta_i^{N+1,K}(t)>e(t)$. By Proposition \ref{pro:capacity_marginal_cost}, $\Delta_i^{N,K}(t)\ge\Delta_i^{N+1,K}(t)>e(t)$, so $i\in\mathcal I_t^{N,K}$. The comparison between $K$ and $K+1$ follows analogously. Hence,
\begin{equation}
\label{eq:capacity_activation_set_comparisons}
\mathcal I_t^{N+1,K}\cap\{1,\ldots,N\}
\subseteq\mathcal I_t^{N,K},\qquad
\mathcal I_t^{N,K+1}
\subseteq\mathcal I_t^{N,K}.
\end{equation}

Under the convention for an empty activation set, the first inclusion implies $i_{N+1,K}^{\min}(t)\ge i_{N,K}^{\min}(t)$. Indeed, if $i_{N+1,K}^{\min}(t)\le N$, then it belongs to $\mathcal I_t^{N,K}$; otherwise, it equals $N+1$ or $N+2$, while $i_{N,K}^{\min}(t)\le N+1$. Since the systems with $K$ and $K+1$ servers have the same state space, the second inclusion directly implies $i_{N,K+1}^{\min}(t)\ge i_{N,K}^{\min}(t)$. This proves the theorem.
\end{proof}

The above theorem translates the marginal-cost comparisons in Proposition \ref{pro:capacity_marginal_cost} into monotonicity of the activation threshold with respect to resource capacities. Increasing $N$ provides a larger inventory buffer, while increasing $K$ provides greater replenishment capacity. Both reduce the marginal continuation cost of an additional depleted unit and thereby increase the number of depleted units required to trigger replenishment.

\section{Cyclic Models}
\label{sec:cyclic_model}

Many replenishment systems face demand and operating costs that vary over time but repeat over daily, weekly, or other cycles. In the EV example, electricity prices follow daily time-of-use (TOU) schedules, while EV arrivals (at a battery-swapping station) vary with the time of day.

Suppose that the planning horizon consists of $M$ cycles of equal length $L$, so that $T=ML$. We assume that the arrival rate and operating cost repeat the same profiles in every cycle. Specifically, there exist functions $\widetilde{\lambda}$ and $\widetilde e$ such that
\begin{equation}
\label{eq:cyclic_parameters}
\lambda((m-1)L+s)=\widetilde{\lambda}(s),\quad
e((m-1)L+s)=\widetilde e(s); \qquad
s\in[0,L),\quad m=1,\ldots,M.
\end{equation}
Similarly, following the definition of $\mathcal I_t$ in (\ref{activate}) and $i_t^{\min}$ in Theorem \ref{theo:threshold}, 
denote the cycle-dependent set of activation states as $\mathcal I_{(m-1)L+s}$, and $i_{(m-1)L+s}^{\min}$ the minimal state thereof. 

Below, in \S\ref{sec:cross_cycle}, we examine how the optimal control changes from one cycle to the next. In \S\ref{sec:peak_offpeak}, we show that the optimal control can reduce to a simple peak-period shutdown and off-peak full activation form under certain conditions, for piecewise-constant $e$.

\subsection{Threshold Ordering across Cycles}
\label{sec:cross_cycle}

\begin{theorem}
\label{theo:cross_cycle}
{\rm
Under \eqref{eq:cyclic_parameters}, for every $m=1,\ldots,M-1$, $i=1,\ldots,N$, and $s\in[0,L]$,
\begin{equation}
\Delta_i((m-1)L+s)\geq\Delta_i(mL+s).
\label{eq:cross_cycle_marginal_cost}
\end{equation}
Consequently, for every $m=1,\ldots,M-1$ and $s\in[0,L)$, we have
$\mathcal I_{(m-1)L+s}\supseteq\mathcal I_{mL+s}$; and hence, $i_{(m-1)L+s}^{\min}\leq i_{mL+s}^{\min}$.
\hfill$\Box$
}
\end{theorem}

The intuition behind the above theorem is straightforward: at the same relative time in two consecutive cycles, the earlier cycle has a longer remaining horizon and faces more future demand, increasing the risk of stockout and making replenishment more valuable.

To derive this ordering from the non-stationary HJB system, the main technical challenge is to compare marginal continuation costs at two times one full cycle apart, e.g., $\Delta_i((m-1)L+s)$ and $\Delta_i(mL+s)$. The marginal costs evolve according to a non-stationary HJB system and do not have explicit-form solutions. Although the model parameters are identical at these two times, the remaining horizons are different. A direct comparison is therefore not available.

The proof in Appendix \ref{app:cross_cycle} will address this challenge through a recursive comparison across cycles. For two consecutive cycles, we subtract the marginal-cost HJB equations evaluated at corresponding times and apply Nagumo's theorem to the resulting difference system. This shows that if $\Delta_i(mL)\geq\Delta_i((m+1)L)$ at the ends of the two cycles, then 
\begin{equation}
\Delta_i((m-1)L+s)\geq\Delta_i(mL+s),\qquad s\in[0,L].
\end{equation}
Starting from $\Delta_i((M-1)L)\geq0=\Delta_i(T)$ and applying this comparison recursively backward across cycles yields \eqref{eq:cross_cycle_marginal_cost}. The full proof is given in Appendix \ref{app:cross_cycle}.

\subsection{Peak--Off-Peak Control}
\label{sec:peak_offpeak}

Piecewise-constant operating costs often follow recurring patterns, such as daily peak and off-peak electricity prices in the EV example.
Let a fraction $\rho\in(0,1)$ of each cycle be designated as the peak period. For $m=1,\ldots,M$, define
\begin{equation}
\begin{aligned}
H_m&=[(m-1)L,(m-1+\rho)L),\\
L_m&=[(m-1+\rho)L,mL).
\end{aligned}
\end{equation}
The operating cost follows
\begin{equation}
e(t)=
\begin{cases}
e_{\mathrm H},&t\in H_m,\\
e_{\mathrm L},&t\in L_m,
\end{cases}
\qquad e_{\mathrm H}>e_{\mathrm L}>0.
\end{equation}
Recall that $\Lambda(t,s)=\int_t^s\lambda(r)\,dr$. By periodicity, the cumulative demand during every peak period is $\Lambda_{\mathrm H}:=\Lambda\bigl(mL,(m+\rho)L\bigr)=\int_0^{\rho L}\widetilde{\lambda}(r)\,dr$ for $m=0,\ldots,M-1$.

In practice, a common way to reduce operating costs under time-varying operating costs (e.g., TOU electricity prices) is load shifting, where operations are shifted from peak to off-peak periods. An extreme form of load shifting shuts down all available servers during peak periods and fully activates them during off-peak periods. We shall call this a ``peak-period shutdown and off-peak full activation'' policy. The following proposition gives these two conditions using model primitives.

\begin{pro}
\label{prop:peak_offpeak}
{\rm
Fix a nonterminal cycle $m\in\{1,\ldots,M-1\}$. If
\begin{equation}
\label{eq:peak_offpeak_condition}
\begin{aligned}
e_{\mathrm H}
&\geq\pi\left(1-e^{-\Lambda((m-1)L,T)}\right)
\geq\Delta_N\bigl((m-1)L\bigr),\\
e_{\mathrm L}
&<\pi\left[1-e^{-\Lambda_{\mathrm H}}
\sum_{k=0}^{N-1}\frac{\Lambda_{\mathrm H}^{k}}{k!}\right]
\leq\Delta_1(mL),
\end{aligned}
\end{equation}
then cycle $m$ exhibits peak-period shutdown and off-peak full activation:
\begin{equation*}
u^*(t,i)=
\begin{cases}
0,&t\in H_m,\\
i\wedge K,&t\in L_m,
\end{cases}
\qquad i=0,\ldots,N.
\qquad\qquad\qquad\qquad \Box
\end{equation*}
}
\end{pro}

The first condition in \eqref{eq:peak_offpeak_condition} requires the peak-period operating cost to be high relative to the future demand exposure from the beginning of cycle $m$ to the end of the horizon. The term $1-e^{-\Lambda((m-1)L,T)}$ is the probability that at least one demand arrives during this remaining horizon. Thus, the peak-period operating cost is high enough that replenishment is not worthwhile even when future demand is likely. This leads to shutdown during the peak periods.

The second condition compares the off-peak operating cost with the probability that at least $N$ demands arrive during $H_{m+1}$, which is the probability that the entire resource pool can be exhausted during that peak period. If the off-peak operating cost is below the expected penalty associated with this event, it is worthwhile to replenish as much as possible before entering the next peak period. This gives full activation during $L_m$.

The first condition is relatively straightforward, 
as the second inequality there 
gives a primitive upper bound on $\Delta_N(t)$, which follows from Lemma \ref{lem:monotonicity} (and its proof in Appendix \ref{app:DC}).
The main difficulty is to obtain a primitive lower bound on $\Delta_1(mL)$ for the second condition. The proof in Appendix \ref{app:cross_cycle} will first use the first condition to show that the system shuts down during $H_{m+1}$. With no replenishment in this period, two systems starting one state apart can be coupled under the same demand arrivals. This gives an exact decomposition of $\Delta_i(mL)$ into the probability of one additional lost demand and a nonnegative continuation term. Dropping the continuation term yields the Poisson-tail lower bound on $\Delta_1(mL)$ used in the second condition.
The full proof is detailed in Appendix \ref{app:cross_cycle}. 

\section{Extension to Phase-type Distributions}
\label{sec:extension}

The exponential replenishment-time assumption makes the number of depleted units a one-dimensional Markov state. We now extend the model to phase-type replenishment times. As discussed in \S4.2.1 (p.~76) of \cite{chenyao}, a phase-type distribution is the absorption time of a finite-state continuous-time Markov chain. Phase-type distributions are dense in the class of distributions on the nonnegative real line under weak convergence; indeed, finite mixtures of Erlang distributions with a common scale parameter already form a dense subclass \cite{asmussen} (Theorem~III.4.2). Thus, the finite-state representation below can approximate general replenishment-time distributions arbitrarily closely while preserving a finite-state continuous-time MDP formulation.

Let $\mathcal J=\{1,\ldots,J\}$ denote the set of transient replenishment phases. A newly depleted unit enters phase $j$ with probability $\alpha_j$, where $\alpha_j\geq0$ and $\sum_{j=1}^J\alpha_j=1$. While a phase-$j$ unit is actively replenished, it moves to phase $\ell\neq j$ at rate $\nu_{j\ell}$ and completes replenishment at rate $\nu_{j0}$. All transition rates are nonnegative, and the resulting transient Markov chain is assumed to reach replenishment completion with probability one. The phase occupied by each depleted unit is observable and remains unchanged whenever replenishment is interrupted. Thus, the control is preemptive-resume at the phase level.

Let $\boldsymbol{x}=(x_1,\ldots,x_J)$ record the numbers of depleted units in the respective phases, and write $i=\sum_{j=1}^Jx_j$. The state and action spaces are
\begin{equation}
\label{eq:ph_state_action}
\begin{aligned}
\mathcal X&=\left\{\boldsymbol{x}\in\mathbb Z_+^J:\sum_{j=1}^Jx_j\leq N\right\},\\
\mathcal U(\boldsymbol{x})&=\left\{\boldsymbol{u}\in\mathbb Z_+^J:0\leq u_j\leq x_j,\ \sum_{j=1}^Ju_j\leq K\right\},
\end{aligned}
\end{equation}
where $u_j$ is the number of phase-$j$ units being replenished. Let $\boldsymbol{e}_j$ be the $j$th unit vector. For $x_j>0$, define the state following a transition from phase $j$ to phase $\ell$ by
\begin{equation}
\label{eq:ph_successor}
S_{j\ell}(\boldsymbol{x})=
\begin{cases}
\boldsymbol{x}-\boldsymbol{e}_j+\boldsymbol{e}_\ell,
&\ell\in\mathcal J\setminus\{j\},\\
\boldsymbol{x}-\boldsymbol{e}_j,
&\ell=0.
\end{cases}
\end{equation}
Accordingly, the nonzero transition rates are
\begin{equation}
\label{eq:ph_transition_rates}
\begin{aligned}
q_{\boldsymbol{x},\boldsymbol{x}+\boldsymbol{e}_j}(t,\boldsymbol{u})
&=\lambda(t)\alpha_j,
&&i<N,\ j\in\mathcal J,\\
q_{\boldsymbol{x},S_{j\ell}(\boldsymbol{x})}(t,\boldsymbol{u})
&=u_j\nu_{j\ell},
&&x_j>0,\ j\in\mathcal J,\
\ell\in\mathcal J\setminus\{j\},\\
q_{\boldsymbol{x},S_{j0}(\boldsymbol{x})}(t,\boldsymbol{u})
&=u_j\nu_{j0},
&&x_j>0,\ j\in \mathcal J.
\end{aligned}
\end{equation}
An arrival in a state with $i=N$ is lost and leaves the state unchanged. The instantaneous cost rate is
\begin{equation}
\label{eq:ph_running_cost}
c(t,\boldsymbol{x},\boldsymbol{u})=e(t)\sum_{j=1}^Ju_j+\pi\lambda(t)\mathbf 1_{\{i=N\}}.
\end{equation}

Let $V(t,\boldsymbol{x})$ denote the minimum expected cost from state $\boldsymbol{x}$ at time $t$. On each continuity interval of $\lambda(t)$ and $e(t)$, its HJB equation is
\begin{eqnarray}
\label{eq:ph_hjb}
&&\quad -\dot V(t,\boldsymbol{x})\,=\, \pi\lambda(t)\mathbf 1_{\{i=N\}}+\lambda(t)\mathbf 1_{\{i<N\}}\sum_{j=1}^J\alpha_j\left[V(t,\boldsymbol{x}+\boldsymbol{e}_j)-V(t,\boldsymbol{x})\right] \\
&&+\min_{\boldsymbol{u}\in\mathcal U(\boldsymbol{x})}\sum_{j:x_j>0}u_j\Big\{e(t)
+\sum_{\ell\in\mathcal J\setminus\{j\}}\nu_{j\ell}\left[V(t,S_{j\ell}(\boldsymbol{x}))-V(t,\boldsymbol{x})\right]  
+\nu_{j0}\left[V(t,S_{j0}(\boldsymbol{x}))-V(t,\boldsymbol{x})\right]\Big\} , \nonumber
\end{eqnarray}
with $V(T,\boldsymbol{x})=0$, for all $\boldsymbol{x}\in\mathcal X$.
Although the optimal action now depends on the complete phase composition, the minimization in \eqref{eq:ph_hjb} admits an exact ranking structure. For each occupied phase, define its phase-specific marginal value rate by
\begin{equation}
\label{eq:ph_ranking_index}
R_j(t,\boldsymbol{x}):=
\sum_{\ell\in\mathcal J\setminus\{j\}}\nu_{j\ell}\left[V(t,\boldsymbol{x})-V(t,S_{j\ell}(\boldsymbol{x}))\right]
+\nu_{j0}\left[V(t,\boldsymbol{x})-V(t,S_{j0}(\boldsymbol{x}))\right],\quad x_j>0.
\end{equation}
Thus, $R_j(t,\boldsymbol{x})$ is the instantaneous expected reduction in continuation cost generated by actively replenishing a phase-$j$ unit, accounting for all of its possible successor phases.

\begin{pro}
\label{pro:phase_type_ranking}
{\rm
At every time-state pair $(t,\boldsymbol{x})$, rank all depleted units in decreasing order of the marginal value rate $R_j(t,\boldsymbol{x})$ associated with their current phases. An optimal Markov policy activates the first
\begin{equation}
\label{eq:ph_number_activated}
\min\left\{K,\sum_{j:x_j>0}x_j\mathbf 1_{\{R_j(t,\boldsymbol{x})>e(t)\}}\right\}
\end{equation}
units in this ranking. Ties may be resolved arbitrarily, and units with $R_j(t,\boldsymbol{x})=e(t)$ are left inactive by convention.
}
\end{pro}

\begin{proof}
By \eqref{eq:ph_ranking_index}, the action-dependent part of the HJB objective is
\begin{equation}
\label{eq:ph_action_objective}
\sum_{j:x_j>0}u_j\left[e(t)-R_j(t,\boldsymbol{x})\right].
\end{equation}
Activating a phase-$j$ unit reduces this objective if and only if $R_j(t,\boldsymbol{x})>e(t)$. Since every activated unit consumes one unit of replenishment capacity, the objective is minimized by activating at most $K$ such units in decreasing order of $R_j(t,\boldsymbol{x})$. The pointwise minimizer is optimal by the same HJB verification argument used for Proposition~\ref{pro:optimal_control}.
\end{proof}

Equivalently, the controller ranks every depleted unit according to the marginal value rate of its current phase and assigns the available servers to the units with the largest rates above the operating-cost rate. When $J=1$ and $\nu_{10}=\mu$, the ranking index reduces to $R_1(t,i)=\mu[V(t,i)-V(t,i-1)]=\mu\Delta_i(t)$. After the normalization $\mu=1$, Proposition~\ref{pro:phase_type_ranking} therefore reduces to the bang--bang policy in Proposition~\ref{pro:optimal_control}.

With multiple phases, however, the total optimal activation $\sum_j u_j^*(t,\boldsymbol{x})$ may lie strictly between $0$ and $i\wedge K$, because only some phases may have marginal value rates above $e(t)$. The resulting structure is therefore a phase-ranking policy rather than the one-dimensional bang--bang policy. The numerical results in \S\ref{sec:numerical_phase_type} examine whether this composition-dependent policy nevertheless exhibits a threshold-like structure after aggregation by the total depleted state $i$.

\section{Numerical Experiments}
\label{sec:numerical}


To validate the theoretical results developed in the previous sections, and in particular to illustrate the robustness of the threshold structure in practical implementations, in this section we report several numerical experiments using a representative set of parameters calibrated from data on EV battery-swapping systems detailed in \S\ref{sec:numerical_setting}.
This is followed by three subsections, where we examine
the optimal charger-activation policy and its performance across different charging-power levels and seasonal electricity-price schedules
in \S\ref{sec:numerical_peak_offpeak},  
and under non-exponential replenishment times
in \S\ref{sec:numerical_phase_type}, 
followed by a brief summary of the key insights garnered from the experiments in \S \ref{sec:numerical_insights}.
 
\subsection{Data and Problem Setup}
\label{sec:numerical_setting}

The demand profile is calibrated using the battery-swapping data reported in \cite{ai,qi}. The hourly EV arrival rate is set to $\lambda_{\rm L}=5.68$ from 0:00 to 8:59 and $\lambda_{\rm H}=14.51$ from 9:00 to 23:59. The electricity-price profile follows Rate Option 1 of Rocky Mountain Power's Electric Vehicle Time-of-Use Pilot (Schedule 2E), as reported by Meredith (2017) \cite{meredith} and adopted in \cite{he}. Under the summer schedule, the on-peak period is 15:00--20:00 for (non-holiday) weekdays; and electricity prices are $c_{\rm H}^{\rm e}=\$0.223/\mathrm{kWh}$ and $c_{\rm L}^{\rm e}=\$0.068/\mathrm{kWh}$ for on-peak and off-peak hours, respectively.

Following \cite{ai,qi}, the baseline-power setting uses a charging power of $P=41$ kW and a mean replenishment time of $1/\mu=0.78$ hours. These values give $\mu\approx1.282$ per hour and operating costs of $e_{\rm H}=P c_{\rm H}^{\rm e}=\$9.143$ and $e_{\rm L}=\$2.788$ per active charger-hour. We set $N=23$ based on the battery capacity of a fourth-generation NIO swapping station \cite{nio}, and set $K=20$ to represent its effective parallel replenishment capacity. The lost-demand penalty is $\pi=\$3.01$ per unserved request, based on the operating data in \cite{liang}.

To examine the optimal control under different operating conditions, we vary the charging power and seasonal electricity-price schedule. Following \cite{qi}, the lower-power setting uses $P=7$ kW and $1/\mu=4.7$ hours, representing slower replenishment under limited grid capacity. These values give $\mu\approx0.213$ per hour, $e_{\rm H}=\$1.561$, and $e_{\rm L}=\$0.476$ per active charger-hour. For the electricity-price profile, we consider both the summer schedule described above and the winter schedule, which adds another on-peak period from 8:00--10:00 \cite{meredith}. Combining the two charging-power levels and two seasonal schedules yields four parameter settings. All other parameters remain at their baseline values. Unless otherwise stated, the experiments use a five-day horizon with the daily demand and electricity-price profiles repeated every 24 hours. 


We solve the finite-state HJB system backward from the terminal condition $V(T,i)=0$, $i=0,\ldots,N$, on a uniform time grid with step size $\Delta t=0.01$ hours. At each grid point, the optimal action is obtained through the pointwise minimization in the HJB equation. The resulting value functions are used to compute $\Delta_i(t)=V(t,i)-V(t,i-1)$ and the corresponding activation threshold according to the condition $\mu\Delta_i(t)>e(t)$.

\subsection{Optimal Charger-Activation Control}
\label{sec:numerical_peak_offpeak}

Implementing the optimal control, which is a time- and state-dependent feedback rule, requires the station to monitor the inventory level of depleted batteries and adjust charger-activation decisions accordingly. Here our focus is on the peak--off-peak rule characterized in Proposition \ref{prop:peak_offpeak}, which has the distinct advantage of simplicity and ease of implementation, and we compare its performance with the fully dynamic optimal policy. 


In Figure \ref{fig:fixed_resource_control}, we illustrate that the simple peak--off-peak rule coincides with the optimal policy over most of the planning horizon. 
Specifically, Figure \ref{fig:fixed_resource_marginals} compares the marginal continuation costs $\Delta_i(t)$ with the expected replenishment cost $e(t)/\mu$.  
Figure \ref{fig:optimal_primitive_control} compares the optimal policy with the actions specified in \eqref{eq:peak_offpeak_condition} of Proposition \ref{prop:peak_offpeak}: during peak-price periods, the optimal threshold is $i_t^{\min}=N+1=24$, implying complete charger shutdown; during (non-terminal) off-peak periods, the threshold becomes $i_t^{\min}=1$, implying full-activation whenever depleted batteries are present. The given data correctly certify the optimality of ``peak-period shutdown and off-peak full-activation'', with the only exception being the final off-peak interval as the planning horizon ends.
(Also note that close to the terminal time, the optimal threshold rises, consistent with Corollary~\ref{cor:increasing_operating_cost}).

\begin{figure}[htbp]
\centering
\begin{subfigure}[t]{0.47\textwidth}
\centering
\includegraphics[width=\textwidth]{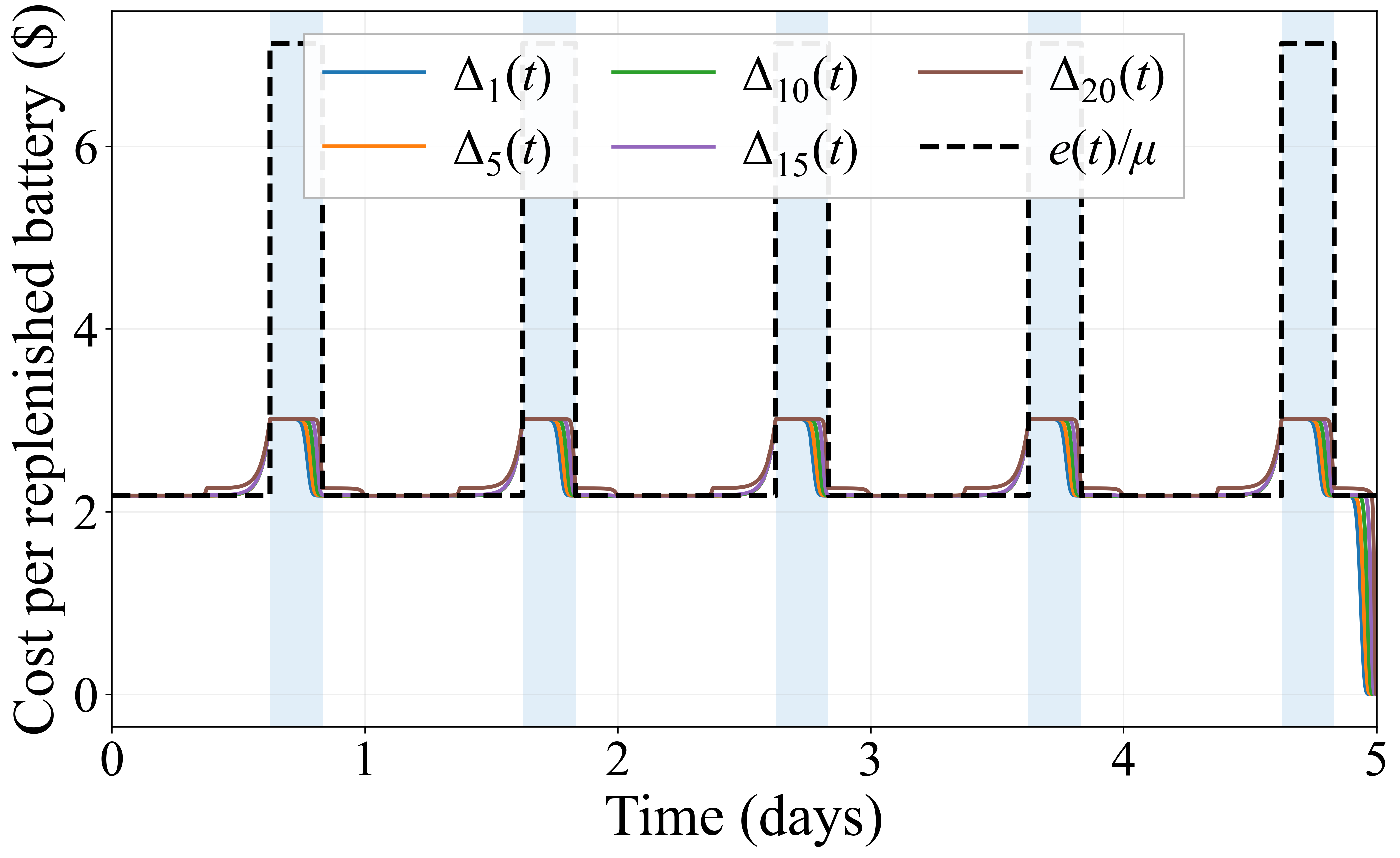}
\caption{Marginal continuation costs and $e(t)/\mu$}
\label{fig:fixed_resource_marginals}
\end{subfigure}
\hfill
\begin{subfigure}[t]{0.47\textwidth}
\centering
\includegraphics[width=\textwidth]{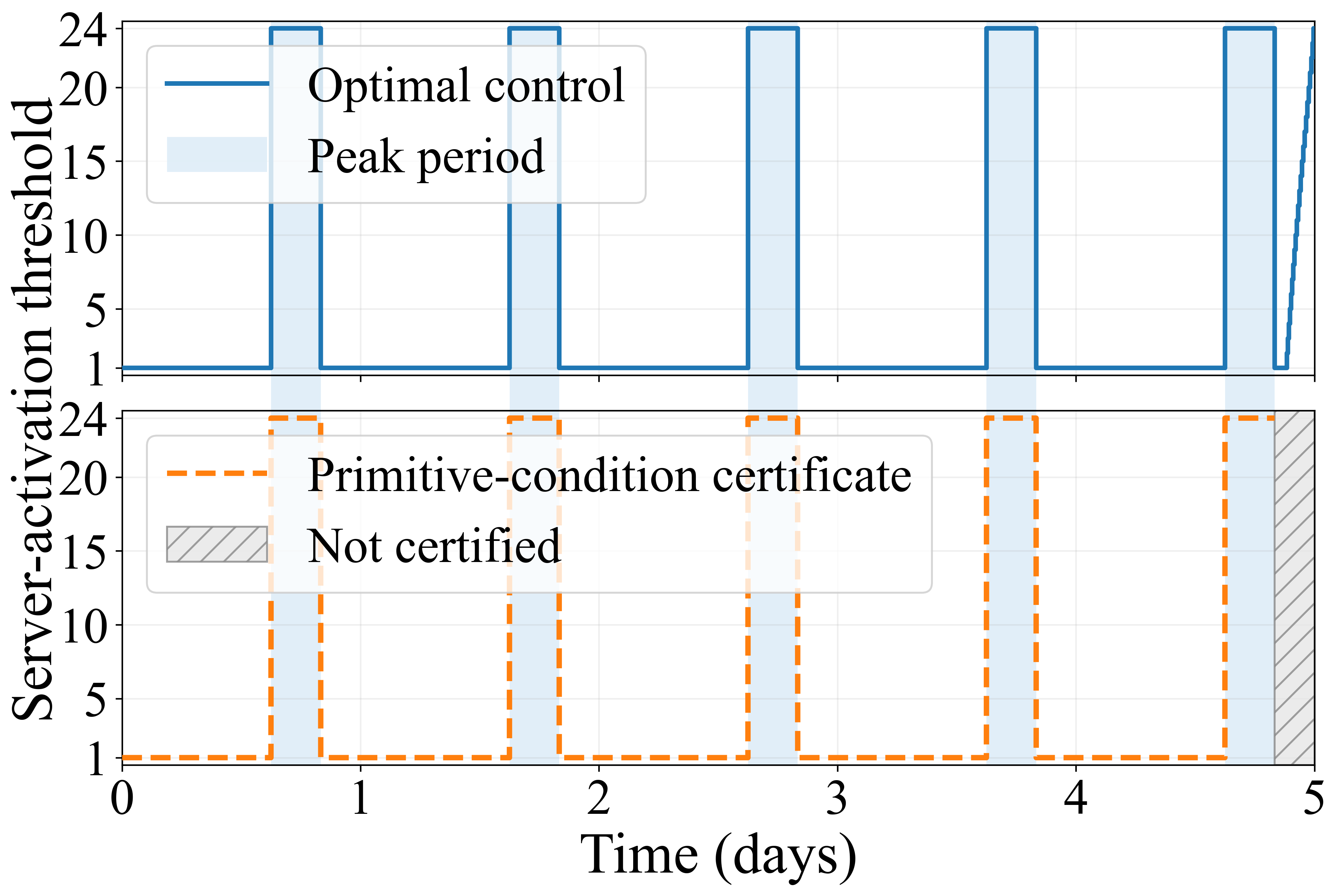}
\caption{Optimal control and primitive certificates}
\label{fig:optimal_primitive_control}
\end{subfigure}
\vspace{1em}
\caption{Optimal peak--off-peak charger activation under the baseline setting.\vspace{0.5em}\protect\\[4pt]
{\footnotesize\textbf{Note:} A threshold of $1$ indicates full feasible activation, whereas $N+1=24$ indicates complete shutdown. The gray-hatched interval in panel (b) is not covered by the primitive sufficient conditions.}}
\label{fig:fixed_resource_control}
\end{figure}


Next, we examine whether the peak--off-peak structure is sensitive to charging power and to electricity-price schedules. In Figure \ref{fig:fixed_resource_control_lp} 
the charging power is reduced from $41$ kW to $7$ kW, which 
changes both the level and the temporal evolution of the marginal continuation costs. Nevertheless, the optimal policy retains the same peak--off-peak structure, and the optimality of the latter continues to be certified by the given data.

\begin{figure}[htbp]
    \centering
    \begin{subfigure}[t]{0.47\textwidth}
        \centering
        \includegraphics[width=\textwidth]
        {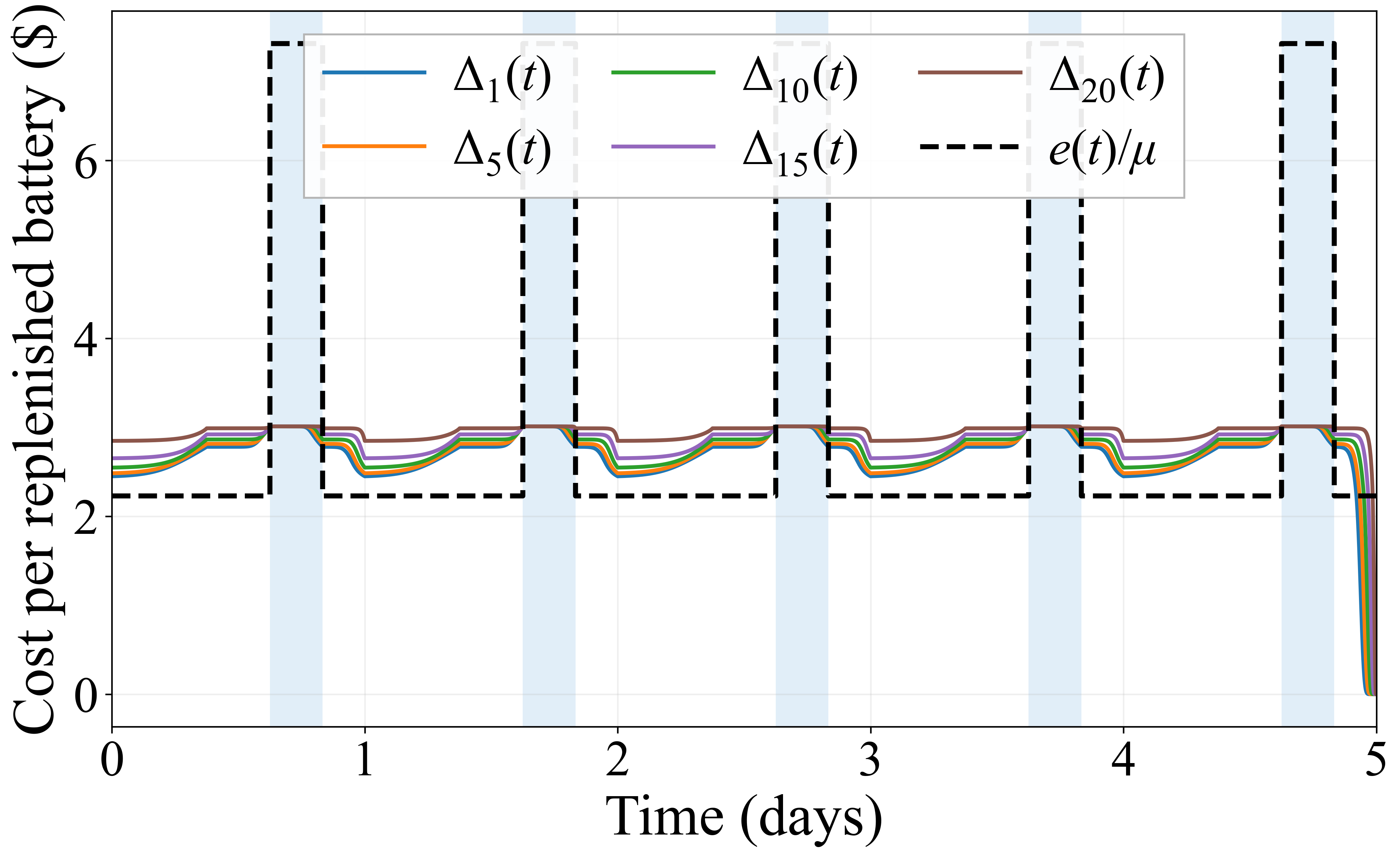}
        \caption{Marginal continuation costs and $e(t)/\mu$}
        \label{fig:fixed_resource_rarginals_lp}
    \end{subfigure}
    \hfill
    \begin{subfigure}[t]{0.47\textwidth}
        \centering
        \includegraphics[width=\textwidth]
        {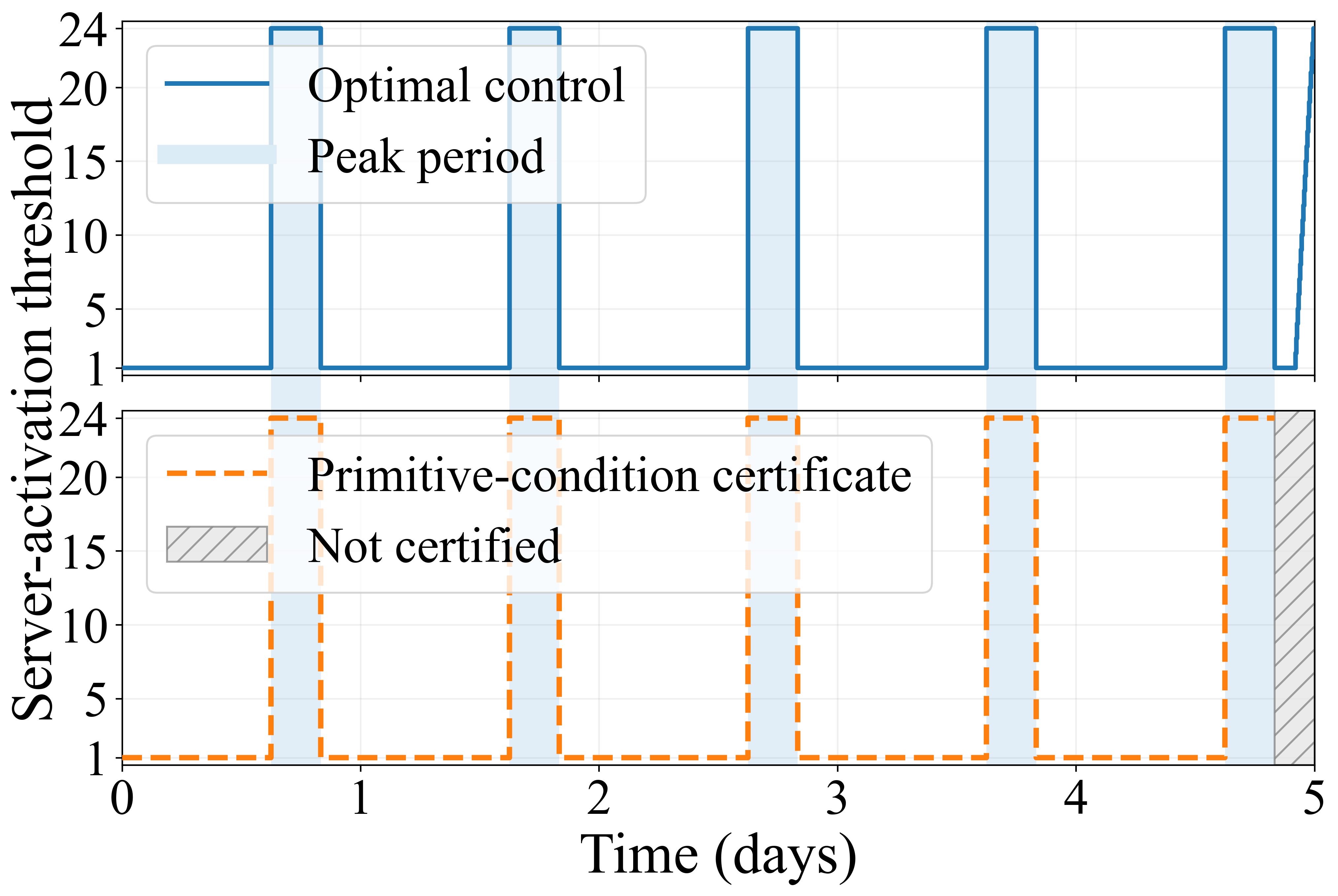}
        \caption{Fully dynamic vs.\ primitive-based control}
        \label{fig:optimal_primitive_control_lp}
    \end{subfigure}
    \vspace{1em}
    \caption{Peak--off-peak charger-activation control with lower charging power and the summer electricity-price schedule.}
    \label{fig:fixed_resource_control_lp}
\end{figure}

Figures \ref{fig:fixed_resource_control_wt} and \ref{fig:fixed_resource_control_wt_lp} consider the winter schedule under the baseline and lower charging powers, respectively. Under both charging-power settings, the optimal policy shuts down all chargers during the 8:00--10:00 morning peak, while the primitive conditions provide only partial certification of the surrounding off-peak full-activation pattern. The partial certification reflects the conservative nature of the primitive sufficient conditions.

\begin{figure}[htbp]
    \centering
    \begin{subfigure}[t]{0.47\textwidth}
        \centering
        \includegraphics[width=\textwidth]
        {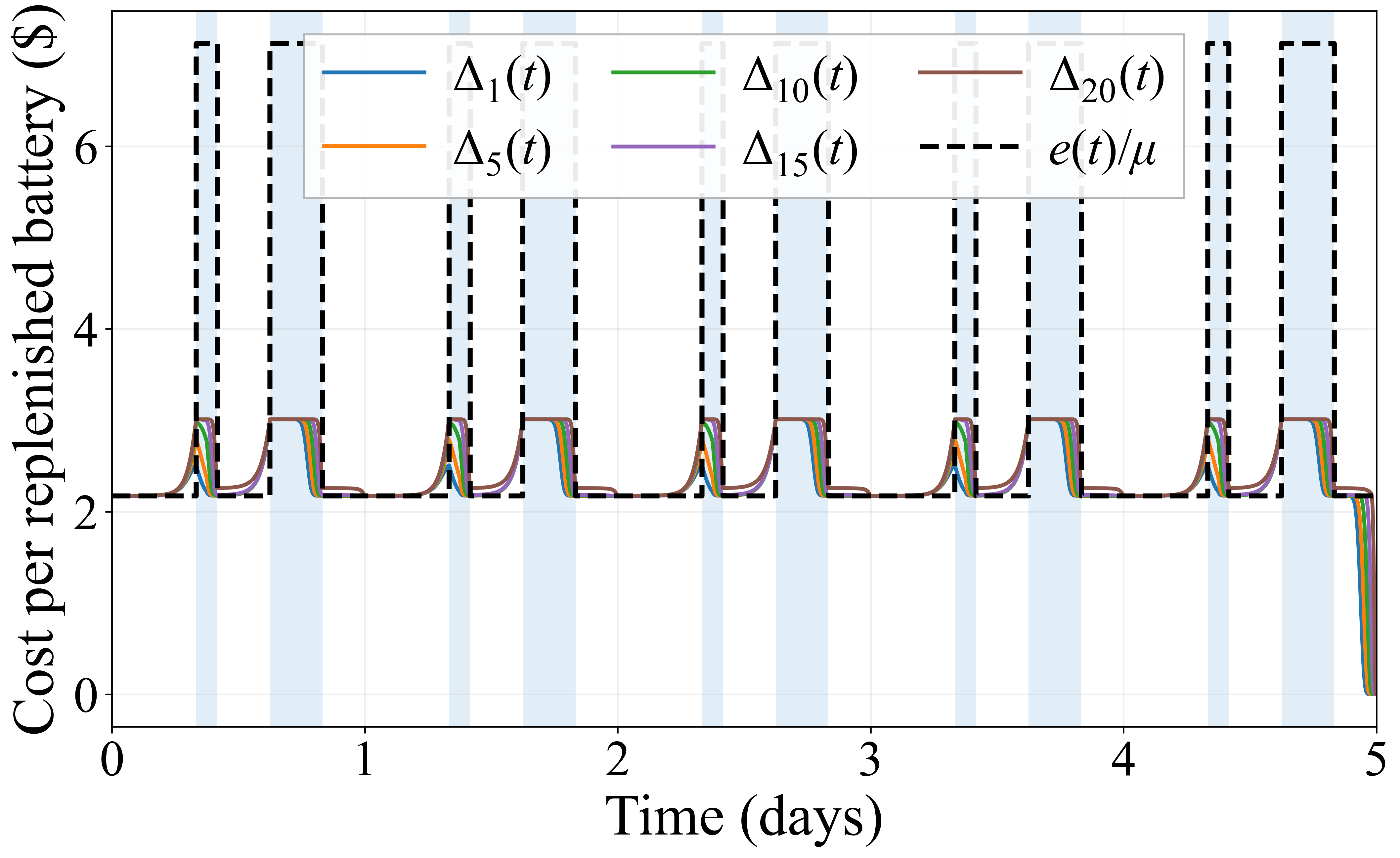}
        \caption{Marginal continuation costs and $e(t)/\mu$}
        \label{fig:fixed_resource_rarginals_wt}
    \end{subfigure}
    \hfill
    \begin{subfigure}[t]{0.47\textwidth}
        \centering
        \includegraphics[width=\textwidth]
        {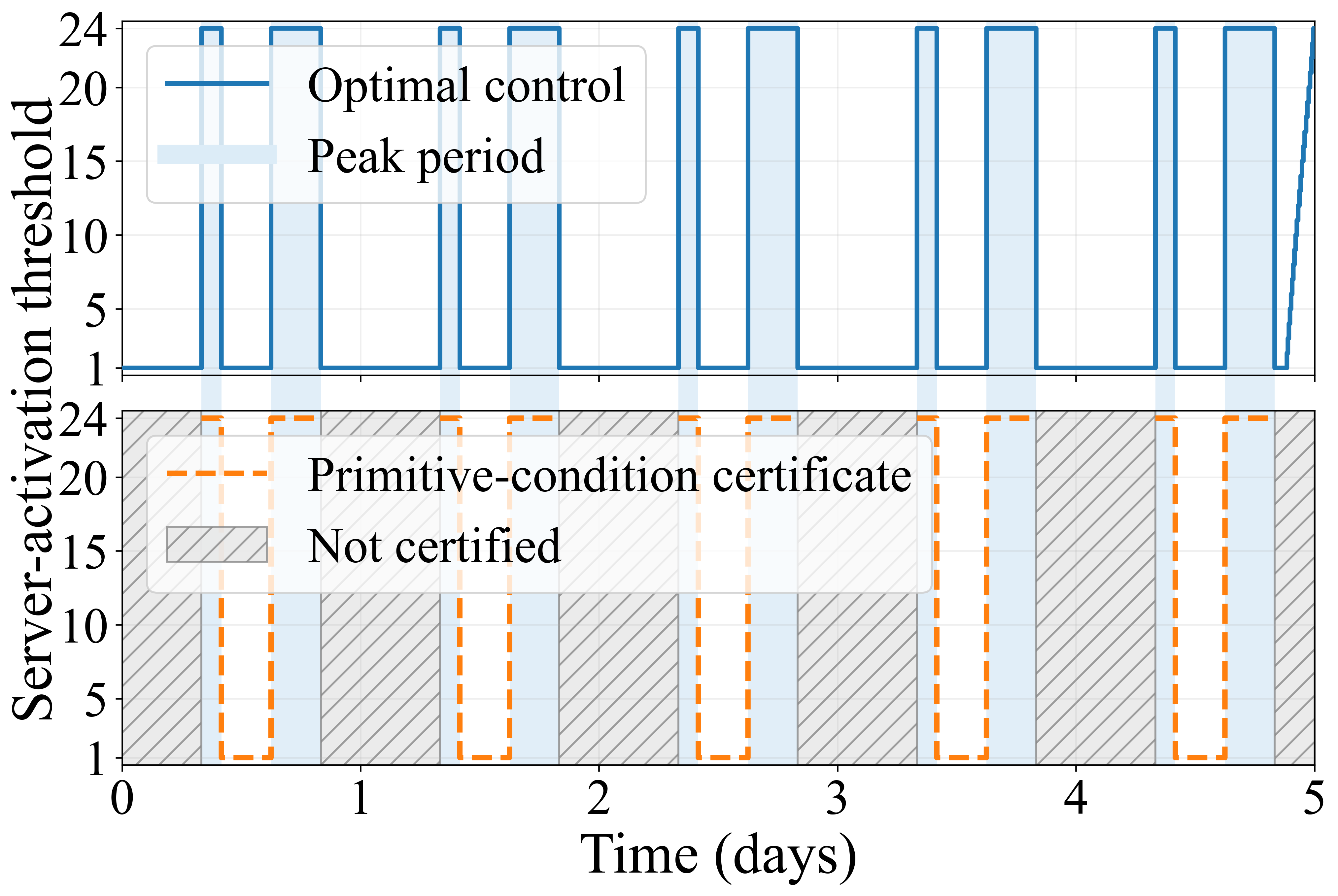}
        \caption{Fully dynamic vs.\ primitive-based control}
        \label{fig:optimal_primitive_control_wt}
    \end{subfigure}
    \vspace{1em}
    \caption{Peak--off-peak charger-activation control with the winter electricity-price schedule and the baseline charging power.}
    \label{fig:fixed_resource_control_wt}
\end{figure}

\begin{figure}[!ht]
    \centering
    \begin{subfigure}[t]{0.47\textwidth}
        \centering
        \includegraphics[width=\textwidth]
        {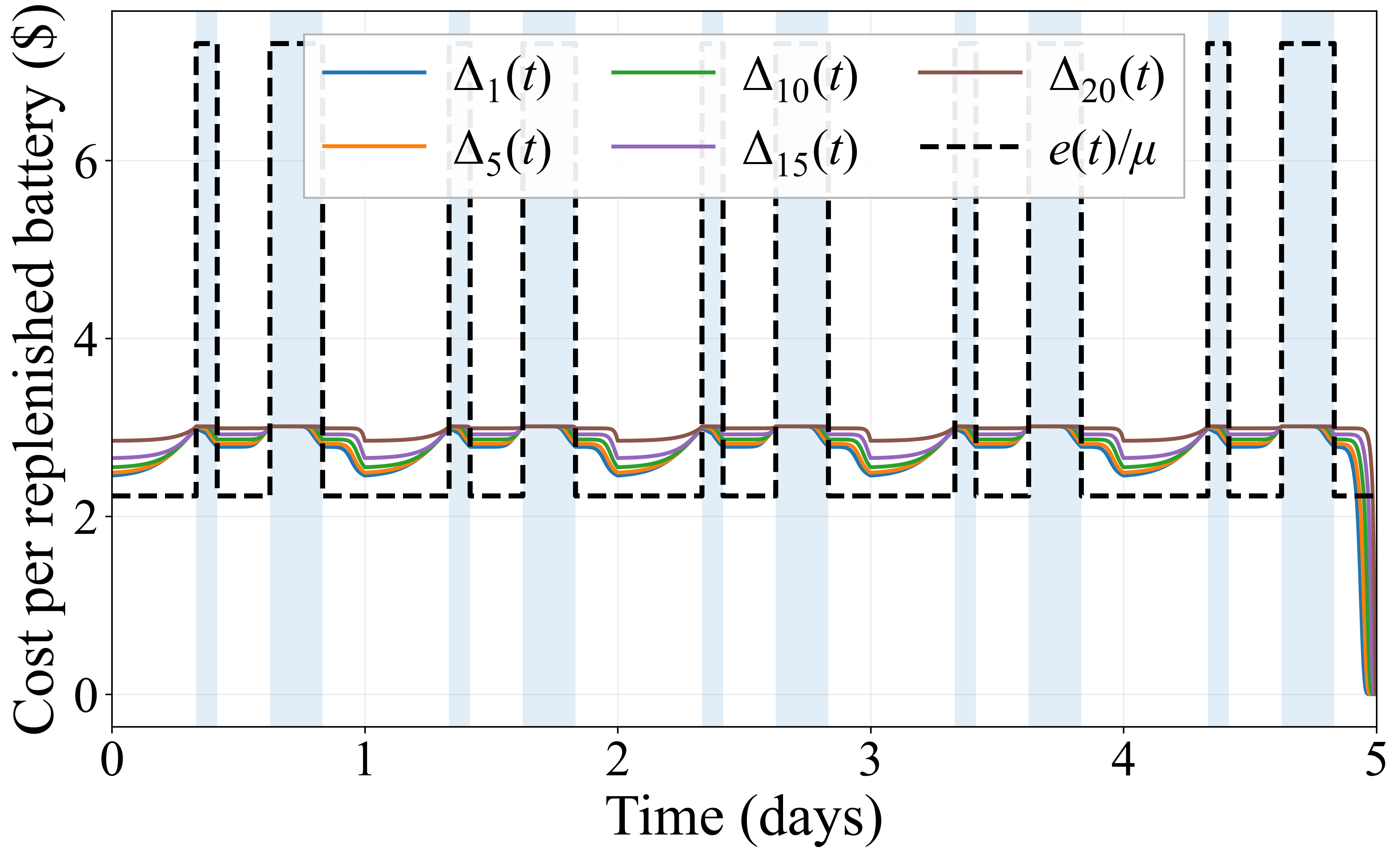}
        \caption{Marginal continuation costs and $e(t)/\mu$}
        \label{fig:fixed_resource_rarginals_wt_lp}
    \end{subfigure}
    \hfill
    \begin{subfigure}[t]{0.47\textwidth}
        \centering
        \includegraphics[width=\textwidth]
        {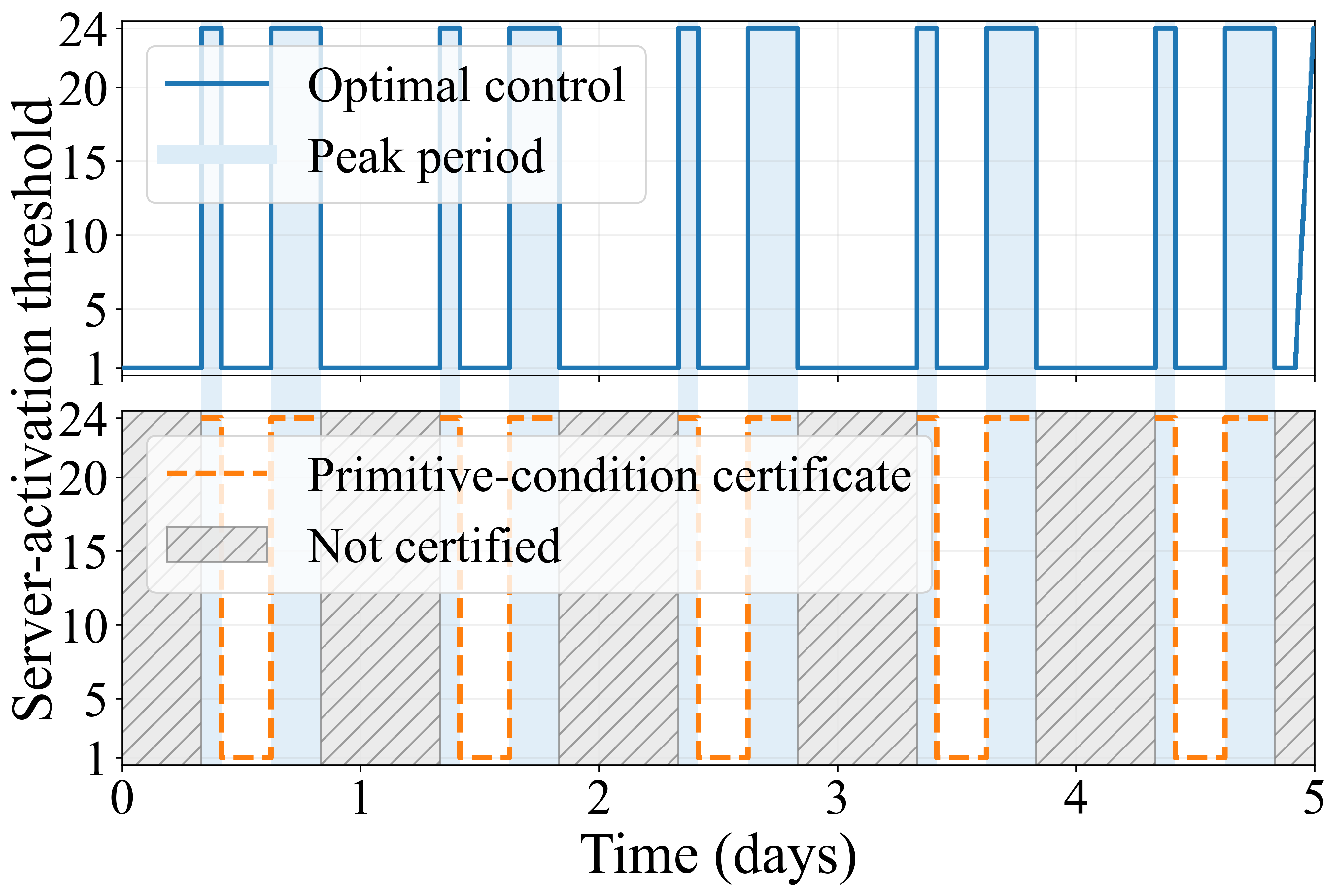}
        \caption{Fully dynamic vs.\ primitive-based control}
        \label{fig:optimal_primitive_control_wt_lp}
    \end{subfigure}
    \vspace{1em}
    \caption{Peak--off-peak charger-activation control under fixed resource capacities with the winter electricity-price schedule and lower charging power.}
    \label{fig:fixed_resource_control_wt_lp}
\end{figure}


\subsection{Non-Exponential Replenishment Times}
\label{sec:numerical_phase_type}

In practice, replenishment/charging times need not follow an exponential distribution. Their variability can be summarized by the squared coefficient of variation (SCV), with the exponential distribution corresponding to SCV $=1$, whereas less variable distributions correspond to SCV $<1$, and more variable distributions to SCV $>1$. Consequently, the optimal control and system performance may depend on the charging-time distribution, as will be demonstrated in the following three subsections. First, we numerically examine whether the optimal control of phase-type models discussed in \S\ref{sec:extension} will exhibit a threshold-like pattern. Second, within the phase-type family, we illustrate how the optimal expected cost varies with the SCV while holding the mean fixed. Third, we study the robustness of the phase-type model serving as a proxy, or {\it approximation}, for general non-exponential charging times. 

\subsubsection{Threshold-Like Optimal Control under Phase-Type Charging Times}

Consider two phase-type charging-time distributions while keeping all other parameters at their baseline values under the summer electricity-price schedule. Both distributions have a mean charging-time of $0.78$ hours. The Erlang-3 ($E_3$) distribution consists of three sequential phases, each with rate $3/0.78$ (hence, SCV$=1/3$). In this case, with the notation of \S\ref{sec:extension}, we have $J=3$, $\alpha_1=1$, $\alpha_2=\alpha_3=0$, and $\nu_{12}=\nu_{23}=\nu_{30}=3/0.78$, with all other transition rates equal to zero. The hyperexponential-2 ($H_2$) distribution is an equal-probability mixture of two exponential branches with respective means $0.26$ and $1.30$ hours (hence, SCV $=17/9$); with $J=2$, $\alpha_1=\alpha_2=1/2$, $\nu_{10}=1/0.26$, $\nu_{20}=1/1.30$, and $\nu_{12}=\nu_{21}=0$.

Starting from $\boldsymbol{X}(0)=\boldsymbol{0}$, let $\varrho_n(\boldsymbol{x})$ denote the proportion of phase state $\boldsymbol{x}$ at time $t_n$ under the optimal policy. For each aggregate depleted state $i$, we classify $(t_n,i)$ as an activation pair if some phase composition with positive proportion induces positive activation. The corresponding projected threshold is
\begin{equation}
\label{eq:ph_projected_threshold}
\begin{aligned}
&u_n^{\mathrm{agg}}(i):=\sum_{\boldsymbol{x}:\,\sum_{j=1}^{J}x_j=i}
\varrho_n(\boldsymbol{x})\sum_{j=1}^{J}u_{n,j}^{*}(\boldsymbol{x}),\\
&i_n^{\mathrm{proj}}:=\min\left\{i\in\{1,\ldots,N\}: u_n^{\mathrm{agg}}(i)>0\right\},
\end{aligned}
\end{equation}
with $i_n^{\mathrm{proj}}=N+1$ when the set is empty.

Figure~\ref{fig:phase_type_aggregated_threshold} reports the resulting activation regions. Under both 
distributions, 
the aggregated positive-activation regions exhibit a clear threshold-like pattern in the number of depleted batteries
(even though the exact optimal action depends on the phase composition). 

\begin{figure}[!ht]
\centering
\includegraphics[width=0.5\textwidth]{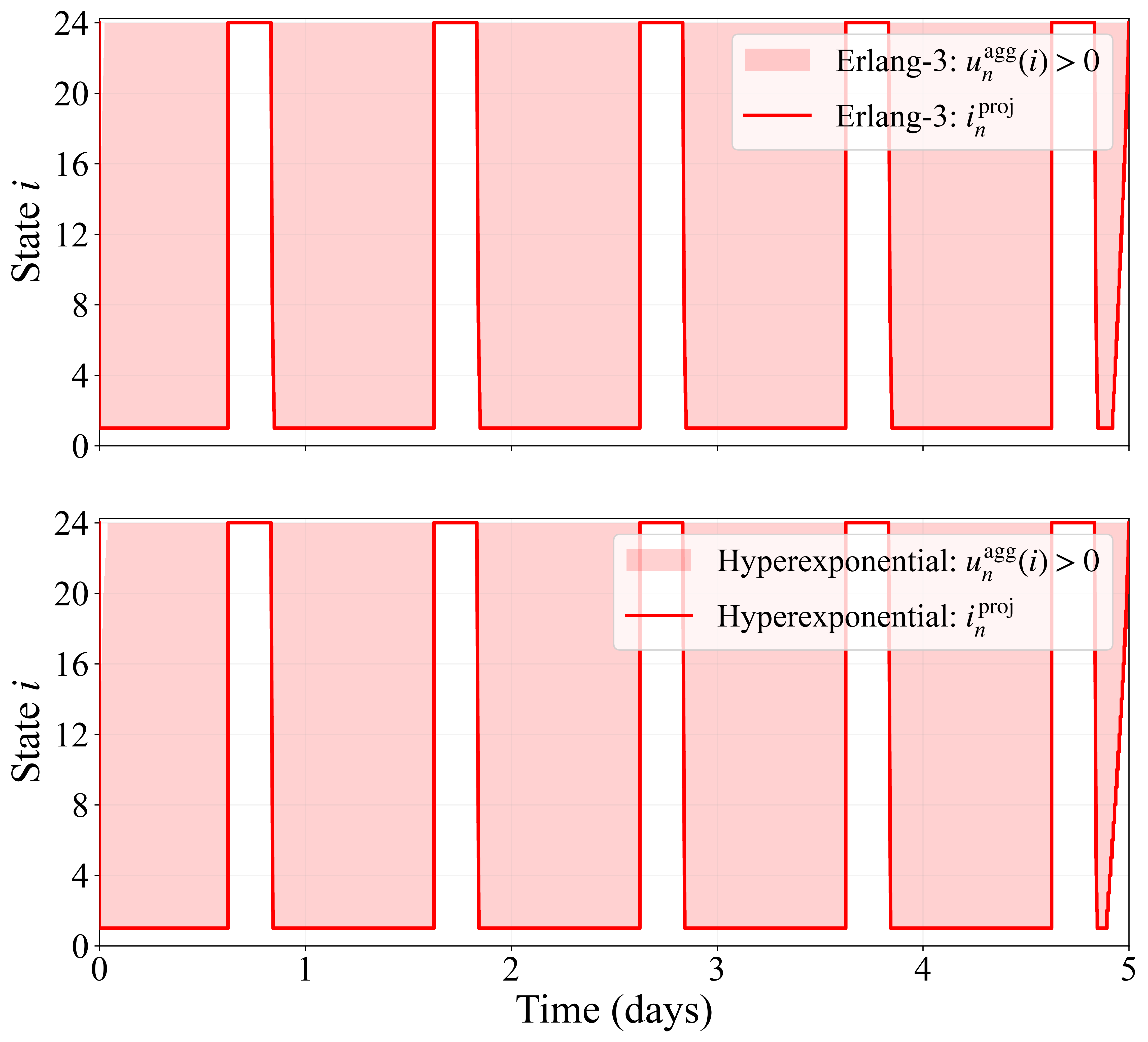}
\vspace{1em}
\caption{Threshold-like aggregated activation under phase-type replenishment times. \vspace{0.5em}\protect\\[4pt]
{\footnotesize\textbf{Note:} The shaded cells indicate aggregate state--time pairs for which at least one phase composition has positive optimal activation, and the red curve denotes their lower boundary $i_n^{\mathrm{proj}}$. A projected threshold of $N+1=24$ indicates no positive activation.}}
\label{fig:phase_type_aggregated_threshold}
\end{figure}

\subsubsection{Impact of Replenishment-Time Variability}

We examine how charging-time variability affects system performance within the phase-type family. All charging-time distributions have the same mean of $m=0.78$ hours. To cover SCVs below one, we consider Erlang-$k$ distributions with $k\in\{2,3,5\}$, phase rate $k/m$, and SCV $1/k$. We use the exponential distribution with rate $1/m$ as the benchmark with SCV one. For SCV $c>1$, we consider moment-matched two-branch hyperexponential distributions ($H_2$). Let $r=\sqrt{(c-1)/(c+1)}$, with branch probabilities $p_f=(1+r)/2$ and $p_s=(1-r)/2$, and corresponding rates $\mu_f=2p_f/m$ and $\mu_s=2p_s/m$. This construction has mean $m$ and SCV $c$ exactly. We consider $c\in\{2,3,5,10\}$. For each distribution, we solve the corresponding phase-composition continuous-time MDP.

Figure~\ref{fig:phase_type_scv_cost} reports the optimal total expected cost and its decomposition into charging and lost-demand costs under two parameter settings. The left panel, Figure~\ref{fig:phase_type_scv_cost_baseline}, reports the baseline results, i.e., for the parameter setting considered above; whereas in the right panel, Figure~\ref{fig:phase_type_scv_cost_high_penalty}, we double the arrival rate and increase the lost-demand penalty from $3.01$ to $6$, while keeping $N=23$ and $K=20$ (same as in the baseline).

\begin{figure}[!ht]
    \centering
    \begin{subfigure}[t]{0.48\textwidth}
        \centering
        \includegraphics[width=\textwidth]{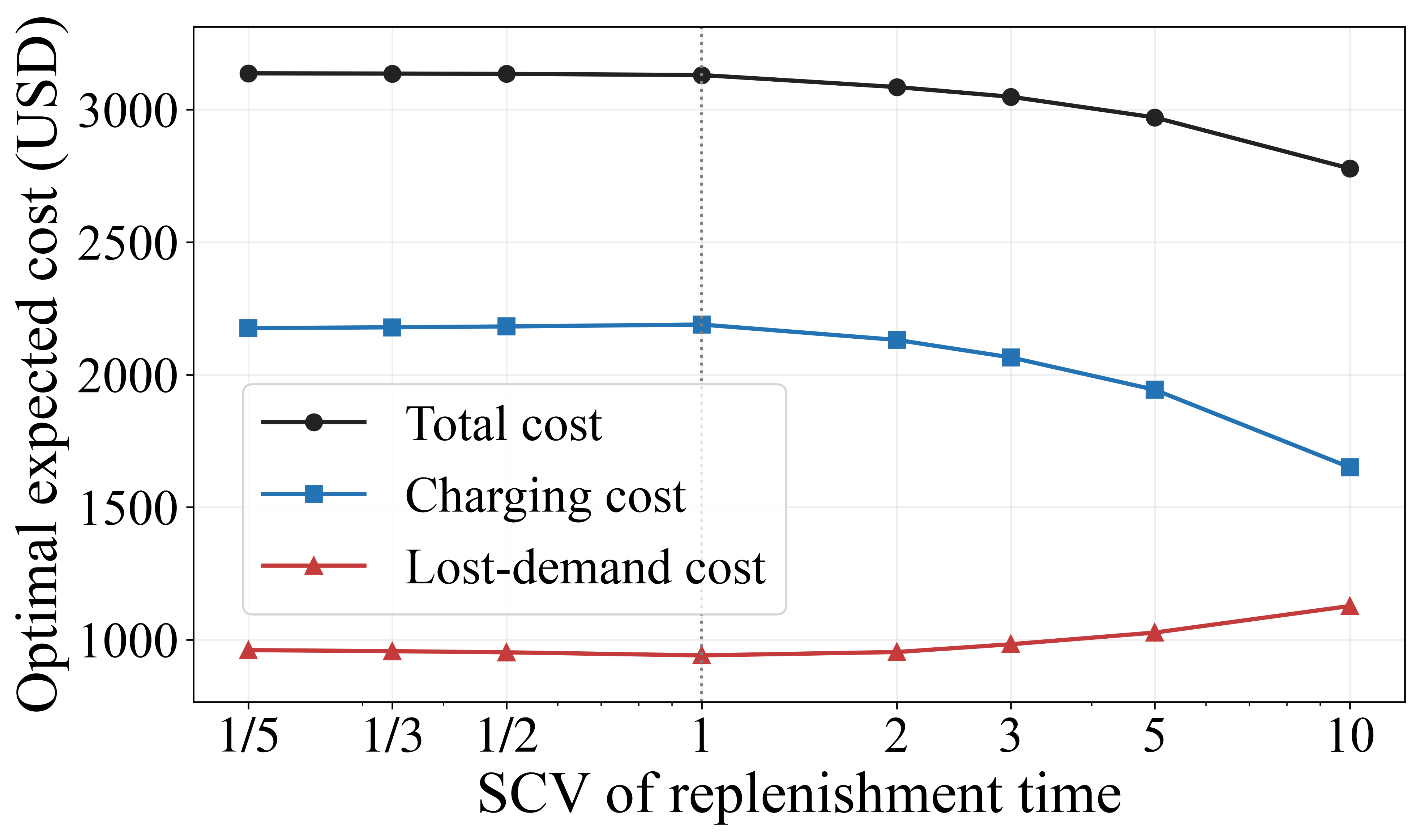}
        \caption{Baseline setting.}
        \label{fig:phase_type_scv_cost_baseline}
    \end{subfigure}
    \hfill
    \begin{subfigure}[t]{0.48\textwidth}
        \centering
        \includegraphics[width=\textwidth]{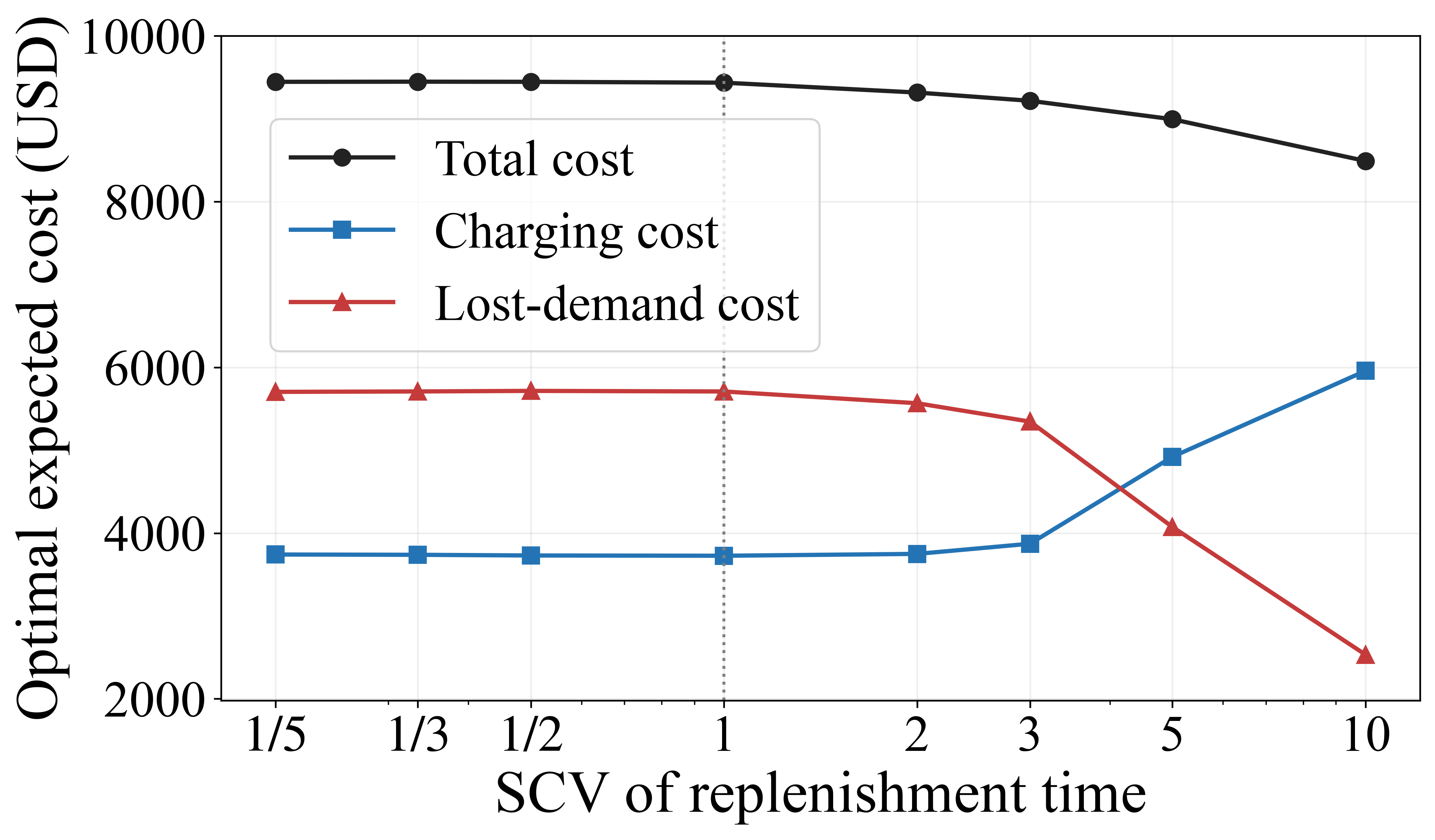}
        \caption{Higher-demand and higher-penalty setting.}
        \label{fig:phase_type_scv_cost_high_penalty}
    \end{subfigure}
    \caption{Optimal expected cost under different phase-type replenishment times. \vspace{0.5em}\protect\\[4pt]
    {\footnotesize\textbf{Note:} The SCV axis is plotted on a logarithmic scale in both panels. (a) uses the baseline parameter setting. (b) doubles both arrival rates and the lost-demand penalty.}}
    \label{fig:phase_type_scv_cost}
\end{figure}

The results in both panels show that for SCV $<1$, deviation of the optimal total cost from 
the exponential benchmark is almost negligible. In other words, the continuous-time MDP control appears to be quite robust for non-exponential charging times with SCV $<1$.

As the SCV increases above $1$, deviations from the exponential benchmark become more pronounced:
 in the left panel, the charging cost decreases whereas the penalty cost increases; by contrast, 
in the right panel, the charging cost increases whereas the penalty cost decreases.
The underlying reason is as follows: 


\begin{itemize}
\item 
In the left panel (baseline case), with the relatively low demand/arrival rate and low penalty cost (for lost demand), the optimal control weighs more heavily on saving charging cost, avoiding batteries with longer charging times, leaving them uncharged even during off-peak hours. This explains the increase in lost demand and penalty cost as the SCV increases.

\item 
In the right panel, with both demand rate and penalty cost doubled, the optimal control shifts its focus so as to avoid losing demand and incurring penalty cost, by charging more aggressively, including charging batteries with shorter charging times even during peak hours. This leads to the decrease in penalty cost, at the price of increased charging cost.

\end{itemize}
In both cases, the decrease in one cost component outweighs the increase in the other, so the overall net effect is the total cost decreasing in the SCV of charging times. 

\subsubsection{Robustness of the Phase-Type Approximation}

We next examine the performance of phase-type approximations when the true charging-time distribution is {\it not} phase type. 
Here we focus on cases with SCV $>1$, where the queueing approximations tend to be more sensitive to service-time variability \cite{johnson1993,johnson1994} (which is similar to what we have observed in the last subsection, although for a different reason --- phase-type as approximation for a general distribution). For computational tractability, we consider a smaller system with $N=4$ batteries and $K=3$ chargers and use a bimodal distribution as the true replenishment-time distribution. To preserve the baseline demand load per battery, both arrival rates are scaled by $4/23$. For each target SCV $c\in\{2,4,6,8,10\}$, the replenishment time has mean $0.78$ hours and is distributed as
\begin{equation}
\label{eq:bimodal_scv_construction}
\mathbb{P}(S=0.39)=\frac{4c}{4c+1},
\qquad
\mathbb{P}\bigl(S=(4c+2)\times 0.39\bigr)=\frac{1}{4c+1}.
\end{equation}
This construction gives $\mathbb{E}[S]=0.78$ hours with SCV $=c$.

We compare two approximations of the bimodal distribution: a) an exponential approximation that matches only the mean $\mathbb{E}[S]=0.78$; and b) an $H_2$ approximation with the same mixture probabilities as the bimodal in (\ref{eq:bimodal_scv_construction}), while replacing the point mass of each branch of the bimodal by an Erlang-2 with the same mean. 
Thus, the $H_2$ approximation matches the bimodal in both the mixture probabilities and the length of the two branches (the latter in mean values, equal to $0.39$ and $(4c+2)\times0.39$ hours, respectively). 
Note that this $H_2$ approximation matches the mean of the bimodal but not the variance; in particular, the $H_2$ has an SCV $=(3c+1)/2$.

We solve the optimal control problem with the charging times following the three distributions --- bimodal, exponential and $H_2$ as specified above --- all using the same decision grid with interval $\Delta=0.195$ hours. 
Figure~\ref{fig:bimodal_approx_optimal_cost} compares the three optimized objective values, and shows that the 
$H_2$ model follows the bimodal quite closely; and both optimized objective values decrease as the SCV increases; in contrast, the exponential model stays constant, as expected, unable to capture the variability of charging times.     

Next, in Figure~\ref{fig:bimodal_approx_policy_cost}, we treat both exponential and $H_2$ as {\it approximations} to the bimodal distribution, by applying their optimal control policies to the bimodal charging times. While the bimodal curve remains the same as in Figure~\ref{fig:bimodal_approx_optimal_cost}, the $H_2$ curve shows a slightly widened gap from the bimodal curve, and the exponential curve remains substantially off from bimodal, even though it now decreases in SCV. The performance of both approximations is as expected: $H_2$ is a good approximation for non-exponential charging times, whereas exponential is not a good approximation when the charging-time variability is high.    

\begin{figure}[!ht]
\centering
\begin{subfigure}[t]{0.47\textwidth}
\centering
\includegraphics[width=\textwidth]
{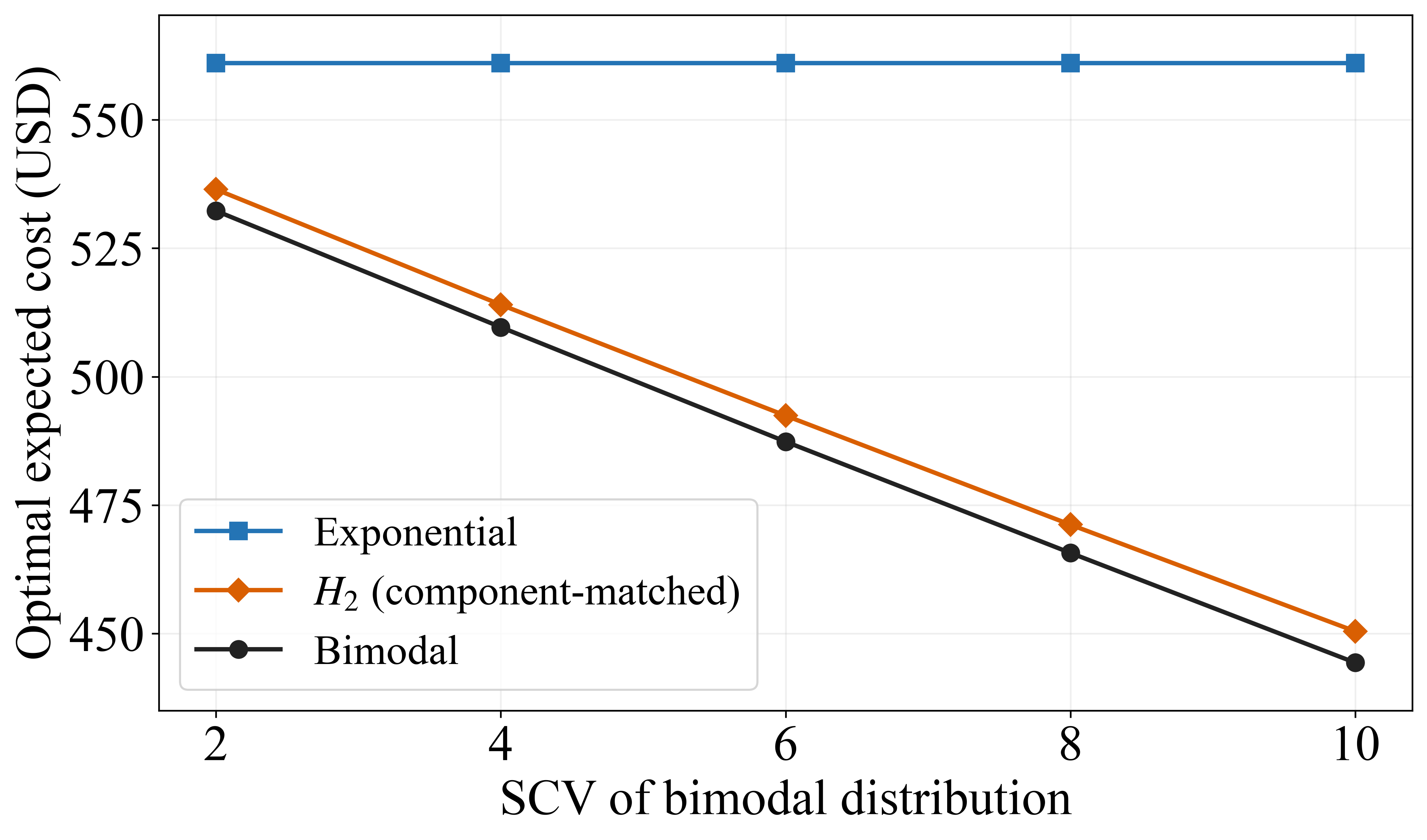}
\caption{Optimal expected costs.}
\label{fig:bimodal_approx_optimal_cost}
\end{subfigure}
\hfill
\begin{subfigure}[t]{0.47\textwidth}
\centering
\includegraphics[width=\textwidth]
{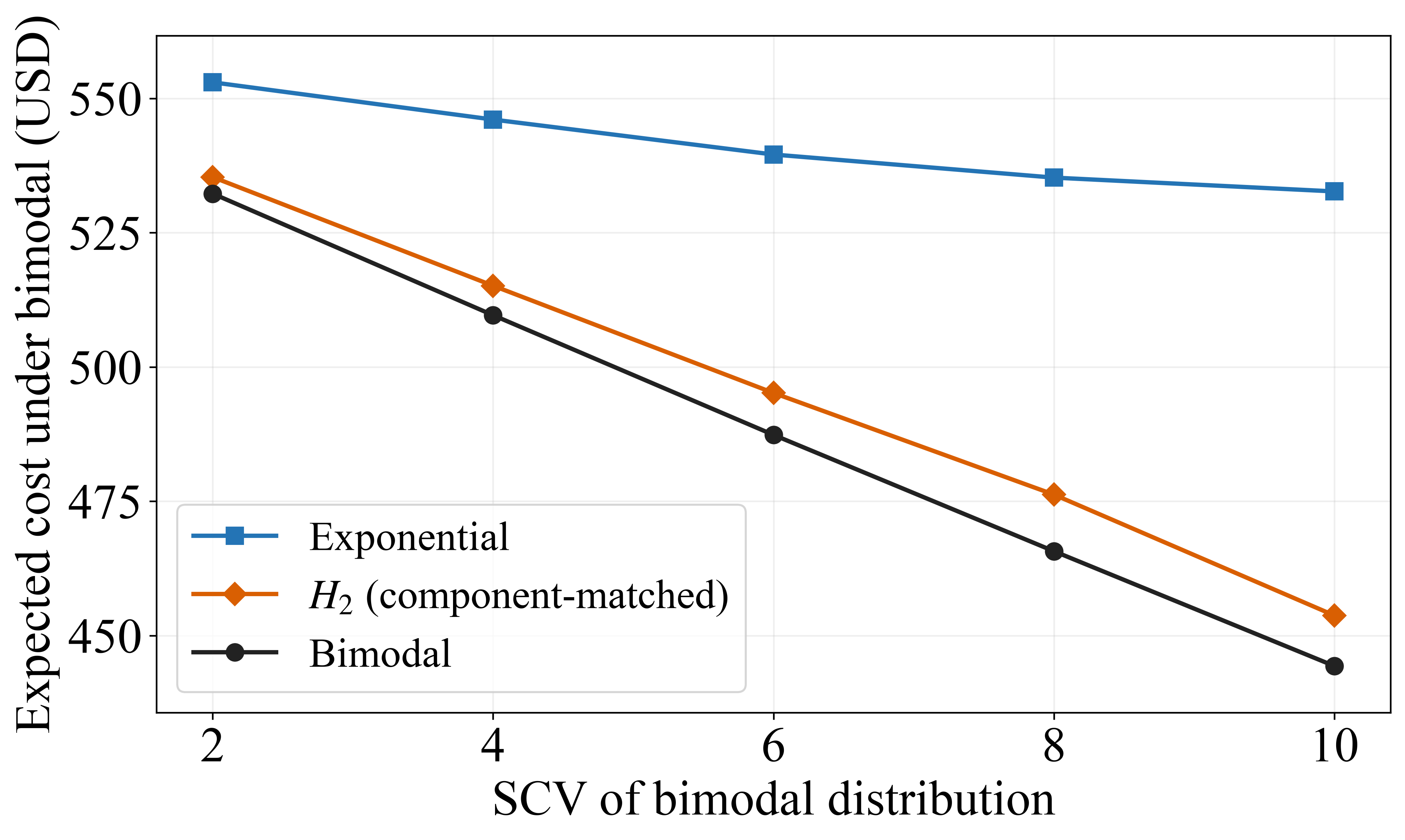}
\caption{Policy costs under the bimodal model.}
\label{fig:bimodal_approx_policy_cost}
\end{subfigure}
\vspace{1em}
\caption{Comparison of the approximations and bimodal models over varying SCVs of the true bimodal distribution.}
\end{figure}

In terms of computational effort, the solution times for both the exponential and $H_2$ approximations remain nearly unchanged across the SCV values considered, 
as shown in Figure~\ref{fig:bimodal_approx_computational_time}.
In contrast, solving the bimodal model exactly becomes increasingly expensive as the SCV increases
(since we need to track the residual charging times in order to derive the optimal policy).

\begin{figure}[!ht]
\centering
\includegraphics[width=0.58\textwidth]
{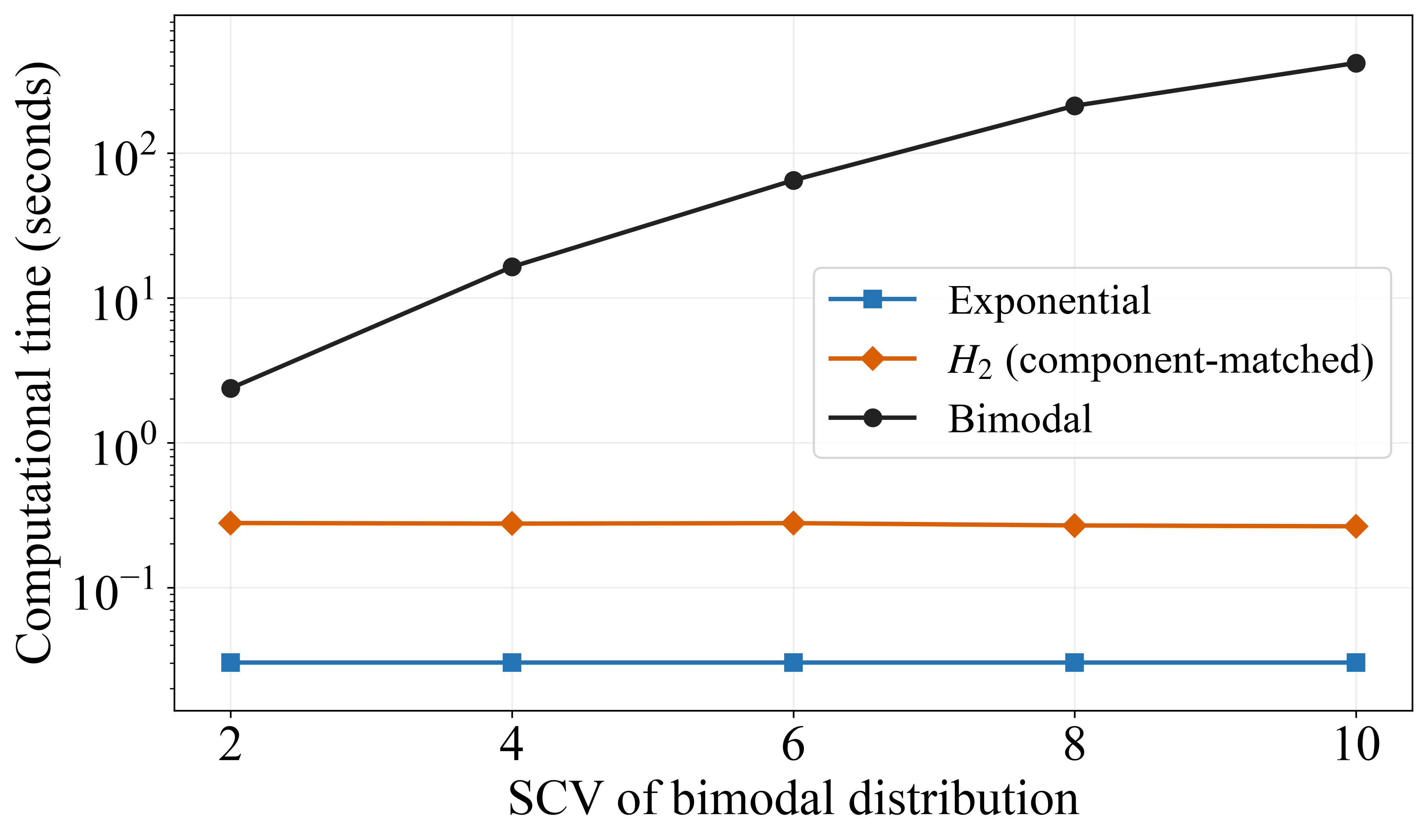}
\caption{Computational times of the approximation and exact bimodal models.}
\label{fig:bimodal_approx_computational_time}
\end{figure}

\subsection{Insights}
\label{sec:numerical_insights}

The experiments provide three main insights. First, under real-world parametric settings, the optimal policy exhibits a clear peak-period shutdown and off-peak full activation structure across different charging-power levels and seasonal electricity-price schedules. The sufficient conditions in Proposition~\ref{prop:peak_offpeak} recover most of this pattern, which supports the robustness of these conditions and illustrates the versatility of the non-stationary Poisson process; in particular, the time-dependent arrival rate (with jumps) appears to compensate well for the limitations of Poisson.

Second, under non-exponential replenishment/charging times, the optimal phase-dependent policy continues to exhibit a clear threshold-like pattern; and more 
importantly, it effectively quantifies the impact of the variability of replenishment times.
Furthermore, the PH distribution, even with a small number of phases (e.g., $H_2$), can serve as a highly accurate approximation for general replenishment-time distributions.
 
Third, the results illustrate how electricity-pricing mechanisms affect the structure of optimal charging policies. Under the TOU (time-of-use) scheme considered here, high peak-period costs induce charger shutdown, whereas low off-peak costs induce full activation. This temporal load shifting is consistent with industry practice, as well as prior studies, e.g., \cite{widrick}. Furthermore, it also resembles the optimal policies under other pricing schemes, such as the one in Li et al. (2026) \cite{li}, which associates in-house electricity (from solar panels installed at the station) with more aggressive charging and grid electricity with more conservative charging.


\section{Conclusions}
\label{sec:conclusion}

In this study we address one of the primary challenges in stochastic control: obtaining exact structural characterizations of the optimal policy in a non-stationary finite-horizon continuous-time MDP setting; particularly when certain rate functions involve jump discontinuities.  Our contributions include a verification of the optimality of bang--bang control using Dynkin's formula, a proof of its threshold structure using Nagumo's theorem, and the derivation of sufficient conditions for the temporal monotonicity of the threshold. 

Motivated by addressing certain operating environments in EV battery-swapping stations  
where the rate functions exhibit cyclic behavior, we further derive a cross-cycle ordering of the thresholds and, under piecewise-constant parameters, sufficient conditions for peak-period shutdown and off-peak (full) activation. To accommodate non-exponential charging-time distributions, we extend the model to phase-type charging times and characterize the optimal control through a phase-ranking policy. Through extensive numerical experiments, we demonstrate the robustness as well as practical effectiveness of the class of threshold policies identified by our study. 




\begin{thebibliography}{99}

\bibitem{ai}
{\sc Ai, W., Cheng, H., and Qi, W.} (2026). Optimizing urban electric vehicle charging and battery swapping
infrastructure: A location-inventory-grid model. {\it Transportation Research Part E: Logistics
and Transportation Review}, 214:104897.

\bibitem{asmussen}
{\sc Asmussen, S.} (2003). {\it Applied Probability and Queues}, vol.\ 51 of Stochastic Modeling and
Applied Probability series. Springer, New York, 2nd edition.

\bibitem{chenyao}
{\sc Chen, H.\ and Yao, D. D.} (2001). {\it Fundamentals of Queueing Networks: Performance, Asymptotics,
and Optimization}. Springer, New York.

\bibitem{davis}
{\sc Davis, J.L., Massey, W.A., and Whitt, W.} (1995). Sensitivity to the service-time distribution in
the nonstationary Erlang loss model. {\it Management Science}, 41(6):1107--1116.

\bibitem{fu}
{\sc Fu, M.C., Marcus, S.I., and Wang, I.-J.} (2000). Monotone optimal policies for a transient queueing
staffing problem. {\it Operations Research}, 48(2):327--331.

\bibitem{harrison}
{\sc George, J.M. and Harrison, J.M.} (2001). Dynamic control of a queue with adjustable service rate.
{\it Operations Research}, 49(5):720-731.

\bibitem{glassermanyao}
{\sc Glasserman, P. and Yao, D.D.} (1994). {\it Monotone Structure in Discrete-Event Systems}, 
Wiley Inter-Science, Series in Probability and Mathematical Statistics, 1994.  

\bibitem{he}
{\sc He, Y., Liu, Z., and Song, Z.} (2022). Integrated charging infrastructure planning and charging
scheduling for battery electric bus systems. {\it Transportation Research Part D: Transport and
Environment}, 111:103437.

\bibitem{jagerman}
{\sc Jagerman, D.L.} (1975). Nonstationary blocking in telephone traffic. {\it Bell System Technical Journal},
54(3):625--661.

\bibitem{johnson1993}
{\sc Johnson, M. A.} (1993).
An empirical study of queueing approximations based on phase-type distributions.
{\it Communications in Statistics---Stochastic Models}, 9(4), 531--561.

\bibitem{johnson1994}
{\sc Johnson, M. A. and Luhman, J. A.} (1994).
Behaviour of queueing approximations based on sample moments.
{\it Applied Stochastic Models and Data Analysis}, 10(4), 233--246.

\bibitem{li}
{\sc Li, Z., Ouyang, H., Sun, Z., and Yuan, Q.} (2026). When peak usage becomes capacity: Optimal charging of battery swapping stations under demand charges. Technical Report 7190119, SSRN.

\bibitem{liang}
{\sc Liang, Y. and Zhang, X.} (2018). Battery swap pricing and charging strategy for electric taxis in China. {\it Energy}, 147:561--577.

\bibitem{massey}
{\sc Massey, W. A. and Whitt, W.} (1996). Stationary-process approximations for the nonstationary Erlang loss model. {\it  Operations Research}, 44(6):976--983.

\bibitem{meredith}
{\sc Meredith, R. M.} (2017). Direct Testimony of Robert M. Meredith. Direct Testimony Docket No. 16-035-36, Rocky Mountain Power. Filed before the Public Service Commission of the State of Utah.

\bibitem{miller}
{\sc Miller, B. M.} (2009). Optimization of queuing system via stochastic control. {\it Automatica}, 45(6):1423-1430.

\bibitem{nagumo}
{\sc Nagumo, M.} (1942). \"Uber die Lage der Integralkurven gew\"ohnlicher Differentialgleichungen. {\it Proceedings
of the physico-mathematical society of Japan}, 3rd Series, 24:551--559.

\bibitem{nio}
{\sc NIO} (2024). NIO Power Swap Station 4.0 Now Operational. NIO Newsroom.

\bibitem{qi}
{\sc Qi, W., Zhang, Y., and Zhang, N.} (2023). Scaling up electric-vehicle battery swapping services in
cities: A joint location and repairable-inventory model. {\it Management Science}, 69(11):6855--6875.

\bibitem{song}
{\sc Song, Y.} (2022). Positive invariance condition for continuous dynamical systems based on Nagumo
theorem. arXiv preprint arXiv:2207.05429.

\bibitem{stidham}
{\sc Stidham, S., Jr.\ and Weber, R.R.} (1989). Monotonic and insensitive optimal policies for control
of queues with undiscounted costs. {\it Operations Research}, 37(4):611--625.


\bibitem{widrick}
{\sc Widrick, R.S., Nurre, S.G., and Robbins, M.J.} (2018). Optimal policies for the management of
an electric vehicle battery swap station. {\it Transportation Science}, 52(1):59--79.

\bibitem{yangyaoye}
{\sc Yang, J., Yao, D.D., and Ye, H.Q.} (2020). On the optimality of reflection control, {\it Operations Research}, 68(6): 1668--1677.


\end{thebibliography}

\newpage

\appendix

\section{Proof of Proposition \ref{pro:optimal_control}.}
\label{app:proof_hjb}

\begin{proof}
Let $W$ be the function constructed in \S\ref{sec:bangbang}. We first show that this construction is well defined and yields a unique function $W$ that is continuous on $[0,T]$, absolutely continuous componentwise, and continuously differentiable on each continuity interval $(\tau_k,\tau_{k+1})$.


Fix $k\in\{0,\ldots,m\}$ and suppose that the right-end value $W(\tau_{k+1})$ has been specified. Writing the HJB system as $-\dot{\mathbf W}(t)=F_k(t,\mathbf W(t))$, extend $e$ and $\lambda$ to the endpoints of $[\tau_k,\tau_{k+1}]$ using their one-sided limits from within the interval. This does not change the HJB equation on the open interval and makes $F_k$ continuous in $t$.

For any $\mathbf w,\mathbf z\in\mathbb R^{N+1}$, we use
\begin{equation*}
\left|\min_u a_u-\min_u b_u\right|
\le\max_u|a_u-b_u|.
\end{equation*}
Together with the bound $u\le K$, this gives
\begin{equation*}
\|F_k(t,\mathbf w)-F_k(t,\mathbf z)\|_\infty
\le 
2\bigl(\lambda(t)+K\bigr)\|\mathbf w-\mathbf z\|_\infty.
\end{equation*}
Since $\lambda$ is bounded on $[0,T]$, $F_k$ is uniformly Lipschitz in its second argument. The Picard--Lindel\"of theorem therefore guarantees a unique solution to the terminal-value problem on $[\tau_k,\tau_{k+1}]$. Moreover, the identity $\dot{\mathbf W}(t)=-F_k(t,\mathbf W(t))$ implies that the derivative has continuous one-sided limits at the endpoints.

Starting from $W(T)=0$ and applying this argument successively backward over the finitely many intervals yields a unique intervalwise $C^1$ solution. Because consecutive pieces are joined using the common value $W(\tau_k)$, the resulting function is continuous on $[0,T]$. Each component is also absolutely continuous, since it is continuous and piecewise $C^1$ over a finite partition.

We next prove that this function coincides with the value function and that a pointwise minimizing Markov policy is optimal.

For $i\ge1$, the terms in the HJB objective that depend on $u$ are $u\bigl(e(t)+W_{i-1}(t)-W_i(t)\bigr)$. Since $\mathcal U(i)=\{0,1,\ldots,\ik\}$, a pointwise minimizing action is
\begin{equation}
u^W(t,i)=
\begin{cases}
\ik, & W_i(t)-W_{i-1}(t)>e(t),\\
0, & W_i(t)-W_{i-1}(t)\le e(t).
\end{cases}
\end{equation}
At equality, every feasible action is a minimizer, and we select $0$ by convention. For $i=0$, the only feasible action is $u^W(t,0)=0$.

Let $u$ be any admissible control and let $I(t)=i$. Since $W$ is $C^1$ on every continuity interval and continuous at each $\tau_k$, Dynkin's formula can be applied separately on these intervals. Summing the resulting identities causes all boundary terms at the discontinuity points to cancel, giving
\begin{equation}
\begin{aligned}
\mathbb E\bigl[W_{I(T)}(T)\bigr]
={}W_i(t)&+\mathbb E\Bigl[\int_t^T(\dot W_{I(s)}(s)+\mathcal L_s^{u(s)}W_{I(s)}(s))ds\Bigr].
\end{aligned}
\end{equation}
Here,
\begin{equation*}
\begin{aligned}
\mathcal L_s^uW_j(s)
=\lambda(s)\bigl(W_{j+1}(s)-W_j(s)\bigr)+u\bigl(W_{j-1}(s)-W_j(s)\bigr).
\end{aligned}
\end{equation*}
By the HJB equation, for every feasible action $v\in\mathcal U(j)$,
\begin{equation}
\dot W_j(s)+\mathcal L_s^vW_j(s)+c(s,j,v)\ge0
\end{equation}
at every continuity point. The finitely many discontinuity points do not affect the integral. Since $W_j(T)=0$, it follows that
\begin{equation}
\begin{aligned}
W_i(t)
&\le\mathbb E\left[\int_t^T c\bigl(s,I(s),u(s)\bigr)\,ds\right]\\
&=J^u(t,i).
\end{aligned}
\end{equation}
As $u$ was arbitrary, $W_i(t)\le V(t,i)$.

Under the Markov policy $u^W$, the pointwise minimum in the HJB equation is attained, so
\begin{equation}
\begin{aligned}
\dot W_j(s)
&+\mathcal L_s^{u^W(s,j)}W_j(s)+c\bigl(s,j,u^W(s,j)\bigr)=0.
\end{aligned}
\end{equation}
Applying the same interval-by-interval Dynkin argument gives $W_i(t)=J^{u^W}(t,i)$. Hence,
\begin{equation}
V(t,i)=\inf_u J^u(t,i)
\le J^{u^W}(t,i)=W_i(t).
\end{equation}
Combining the two inequalities yields
\begin{equation}
W_i(t)=V(t,i)=J^{u^W}(t,i),
\qquad (t,i)\in[0,T]\times\mathcal S.
\end{equation}
Thus, $u^W$ is optimal.

Finally, since $W_i(t)=V(t,i)$, we have $W_i(t)-W_{i-1}(t)=\Delta_i(t)$. Therefore, the selected optimal Markov policy is
\begin{equation}
u^*(t,i)=
\begin{cases}
\ik, & \Delta_i(t)>e(t),\\
0, & \Delta_i(t)\le e(t),
\end{cases}
\end{equation}
which proves the asserted bang--bang characterization.
\end{proof}

\section{Proof of Lemma \ref{lem:monotonicity}}
\label{app:DC}

\begin{proof}
Let $u_i:=i\wedge K$, $i=1,\ldots,N$, and define
\begin{equation}
\begin{aligned}
\Psi_i(t)&:=u_i\min\{0,e(t)-\Delta_i(t)\},\qquad i=1,\ldots,N,
\end{aligned}
\end{equation}
with $\Psi_0(t):=0$. Using $\Delta_i(t)=V(t,i)-V(t,i-1)$,
the HJB system can be written as
\begin{equation}
\begin{aligned}
-\frac{d}{dt}V(t,i)
=\pi\lambda(t)\mathbf 1_{\{i=N\}}
+\lambda(t)\Delta_{i+1}(t)+\Psi_i(t),\qquad i=0,\ldots,N,
\end{aligned}
\end{equation}
where $\Delta_{N+1}(t):=0$. Taking first differences gives
\begin{equation}
\label{eq:threshold_marginal_system}
\begin{aligned}
-\dot\Delta_1(t)&=\lambda(t)\bigl(\Delta_2(t)-\Delta_1(t)\bigr)+\Psi_1(t),\\
-\dot\Delta_i(t)&=\lambda(t)\bigl(\Delta_{i+1}(t)-\Delta_i(t)\bigr)+\Psi_i(t)-\Psi_{i-1}(t),\qquad i=2,\ldots,N-1,\\
-\dot\Delta_N(t)&=\lambda(t)\bigl(\pi-\Delta_N(t)\bigr)+\Psi_N(t)-\Psi_{N-1}(t).
\end{aligned}
\end{equation}
The terminal condition implies $\Delta_i(T)=0$, $i=1,\ldots,N$.

We first establish the weak inequalities in (\ref{delta1}). Reverse time by
setting $\tau=T-t$. To simplify notation, we continue to write
$\Delta_i,\Psi_i,e$, and $\lambda$ for the time-reversed functions, and
use a prime to denote differentiation with respect to $\tau$. Then
(\ref{eq:threshold_marginal_system}) becomes
\begin{equation}
\label{eq:threshold_backward_marginal_system}
\begin{aligned}
\Delta_1'&=\lambda(\Delta_2-\Delta_1)+\Psi_1,\\
\Delta_i'&=\lambda(\Delta_{i+1}-\Delta_i)+\Psi_i-\Psi_{i-1},\qquad i=2,\ldots,N-1,\\
\Delta_N'&=\lambda(\pi-\Delta_N)+\Psi_N-\Psi_{N-1},
\end{aligned}
\end{equation}
with $\Delta_i(0)=0$.

Consider the region $\mathcal C:=\left\{(\Delta_1,\ldots,\Delta_N):0\le\Delta_1\le\cdots\le\Delta_N\le\pi\right\}$. We verify that $\mathcal C$ is invariant under the backward-time dynamics. On the boundary face $\Delta_1=0$, the positivity of $e$ gives $\Psi_1=0$, and hence $\Delta_1'=\lambda\Delta_2\ge0$.

For $i=2,\ldots,N$, let $\delta_i:=\Delta_i-\Delta_{i-1}$. At a boundary point where $\delta_i=0$, we have $\Delta_i=\Delta_{i-1}$. Let $f(x):=\min\{0,e-x\}$. Since $f$ is decreasing and nonpositive, for $i\ge3$, we have $f(\Delta_{i-2})\ge f(\Delta_{i-1})=f(\Delta_i)$. Because $\Psi_j=u_j f(\Delta_j)$, it follows that
\begin{equation}
\Psi_i-2\Psi_{i-1}+\Psi_{i-2}\ge (u_i-2u_{i-1}+u_{i-2})f(\Delta_i)\ge0.
\label{eq:threshold_psi_second}
\end{equation}
The last inequality holds because $u_i=i\wedge K$ is discretely concave
and $f(\Delta_i)\le0$. For $i=2$, the same conclusion follows directly
from $u_0=\Psi_0=0$: when $\Delta_2=\Delta_1$,
\begin{equation}
\Psi_2-2\Psi_1+\Psi_0
=(u_2-2u_1)f(\Delta_2)\ge0.
\end{equation}

Subtracting adjacent equations in
(\ref{eq:threshold_backward_marginal_system}) gives
\begin{equation}
\label{eq:threshold_backward_gap_system}
\begin{aligned}
\delta_i'&=\lambda(\delta_{i+1}-\delta_i)+\Psi_i-2\Psi_{i-1}+\Psi_{i-2},\qquad i=2,\ldots,N-1,\\
\delta_N'&=\lambda(\pi-\Delta_N)-\lambda\delta_N+\Psi_N-2\Psi_{N-1}+\Psi_{N-2}.
\end{aligned}
\end{equation}
Therefore, on the boundary face $\delta_i=0$,
\begin{equation}
\label{eq:threshold_gap_boundary}
\delta_i'\ge
\begin{cases}
\lambda\delta_{i+1}\ge0, & i=2,\ldots,N-1,\\
\lambda(\pi-\Delta_N)\ge0, & i=N.
\end{cases}
\end{equation}

It remains to consider the upper boundary $\Delta_N=\pi$. Within $\mathcal C$, $\Delta_N\ge\Delta_{N-1}$. Since $f$ is decreasing and nonpositive and $u_N\ge u_{N-1}$,
\begin{equation*}
\Psi_N=u_Nf(\Delta_N)
\le u_{N-1}f(\Delta_{N-1})=\Psi_{N-1}.
\end{equation*}
Consequently, whenever $\Delta_N=\pi$,
\begin{equation}
\label{eq:threshold_upper_boundary}
\Delta_N'=\Psi_N-\Psi_{N-1}\le0.
\end{equation}

The set $\mathcal C$, defined by finitely many linear inequalities, is a closed convex polyhedron. By Nagumo's theorem \cite{nagumo}, or its polyhedral characterization in Theorem~3.7 of \cite{song}, its invariance follows if (i) the backward-time dynamics have a unique solution and (ii) the vector field satisfies the tangent condition on every active boundary face. The latter has been verified above: on $\Delta_1=0$, \eqref{eq:threshold_backward_marginal_system} gives $\Delta_1'=\lambda\Delta_2\ge0$; on $\Delta_i-\Delta_{i-1}=0$, \eqref{eq:threshold_gap_boundary} gives $\delta_i'\ge0$; and on $\Delta_N=\pi$, \eqref{eq:threshold_upper_boundary} gives $\Delta_N'\le0$.

Moreover, the function $f(x)=\min\{0,e-x\}$ defined above is $1$-Lipschitz. Since $\Psi_i=u_i f(\Delta_i)$ and all the other terms in \eqref{eq:threshold_backward_marginal_system} are linear in the marginal costs, the right-hand side of \eqref{eq:threshold_backward_marginal_system} is Lipschitz in $(\Delta_1,\ldots,\Delta_N)$ on each continuity interval of $e$ and $\lambda$. The Picard--Lindel\"of theorem therefore guarantees a unique solution to the backward-time system for any given initial condition. Since $\Delta_i(0)=0$ for all $i$, the backward-time trajectory starts in $\mathcal C$. Nagumo's theorem therefore implies that the trajectory remains in $\mathcal C$.
Thus,
\begin{equation}
0\le\Delta_1(t)\le\Delta_2(t)\le\cdots
\le\Delta_N(t)\le\pi,
\quad t\in[0,T].
\end{equation}
This argument applies on every interval between consecutive discontinuity points of $e$ and $\lambda$. The value function, and hence its marginal differences, is continuous at these points, so the inequalities are preserved when the intervalwise solutions are joined.

We next strengthen the upper bound for $\Delta_N$. The weak ordering and the argument above imply $\Psi_N\le\Psi_{N-1}$. Hence, in backward time,
\begin{equation}
\Delta_N'\le\lambda(\pi-\Delta_N).
\end{equation}
Since $\Delta_N(0)=0$, the scalar comparison theorem gives, in the original time variable,
\begin{equation}
\label{eq:threshold_delta_N_upper_bound}
\begin{aligned}
\Delta_N(t)
&\le \pi\left(1-\exp\left\{-\int_t^T\lambda(s)\,ds\right\}\right)<\pi
\end{aligned}
\end{equation}
for $t<T$, where the strict inequality follows from $\inf_{s\in[0,T]}\lambda(s)>0$.

We finally prove the strict ordering. Continue to work in backward time. If $\delta_N(\tau_0)=0$ for some $\tau_0>0$, then the nonnegative function $\delta_N$ attains its minimum at $\tau_0$. However,
(\ref{eq:threshold_psi_second}),
(\ref{eq:threshold_backward_gap_system}), and
(\ref{eq:threshold_delta_N_upper_bound}) imply
\begin{equation}
\delta_N'(\tau_0)
\ge\lambda(\tau_0)\bigl(\pi-\Delta_N(\tau_0)\bigr)>0,
\end{equation}
contradicting the first-order condition at an interior minimum. If $\tau_0$ is a discontinuity point of the primitives or the endpoint of an interval, the same contradiction follows from the left derivative: at a minimum it is nonpositive, whereas the backward dynamics make it strictly positive. Therefore,
\begin{equation}
\delta_N(\tau)>0,
\qquad \tau>0.
\end{equation}

Now suppose that $\delta_{i+1}(\tau)>0$ for every $\tau>0$, where $i\in\{2,\ldots,N-1\}$. If $\delta_i(\tau_0)=0$ for some $\tau_0>0$, then (\ref{eq:threshold_psi_second}) and (\ref{eq:threshold_backward_gap_system}) give
\begin{equation*}
\delta_i'(\tau_0)
\ge\lambda(\tau_0)\delta_{i+1}(\tau_0)>0,
\end{equation*}
which yields the same contradiction. Backward induction over $i=N-1,N-2,\ldots,2$ therefore gives
\begin{equation}
\delta_i(t)>0,
\qquad i=2,\ldots,N,\quad t<T.
\end{equation}

Returning to the original time variable, we obtain
\begin{equation}
0\le\Delta_1(t)<\Delta_2(t)<\cdots<\Delta_N(t)<\pi,
\quad t\in[0,T),
\end{equation}
whereas $\Delta_i(T)=0$ for every $i=1,\ldots,N$. Therefore, the value function is discretely convex in $i$, and strictly so for $t<T$.
\end{proof}

%

\section{Proof of Lemma \ref{lem:lower_bound_D}}
\label{app:time-monotonicity}

\begin{proof}
Let $u_i:=i\wedge K$ and $\Psi_i(t):=u_i\min\{0,e(t)-\Delta_i(t)\}$, with $u_0:=0$ and $\Psi_0(t):=0$. Recall that $\delta_i(t):=\Delta_i(t)-\Delta_{i-1}(t)$. Taking adjacent differences in \eqref{eq:threshold_marginal_system} gives
\begin{equation}
\label{eq:delta_gap_dynamics}
\begin{aligned}
-\dot\delta_i(t)&=\lambda(t)\bigl(\delta_{i+1}(t)-\delta_i(t)\bigr)+\Psi_i(t)-2\Psi_{i-1}(t)+\Psi_{i-2}(t),\qquad i=2,\ldots,N-1,\\
-\dot\delta_N(t)&=\lambda(t)\bigl(\pi-\Delta_N(t)-\delta_N(t)\bigr)+\Psi_N(t)-2\Psi_{N-1}(t)+\Psi_{N-2}(t),
\end{aligned}
\end{equation}
with $\delta_i(T)=0$. Fix $t$ and let $f_t(x):=\min\{0,e(t)-x\}$. Since $f_t$ is decreasing and $1$-Lipschitz, Lemma \ref{lem:monotonicity} implies $f_t(\Delta_i(t))-f_t(\Delta_{i-1}(t))\ge-\delta_i(t)$ and $f_t(\Delta_{i-2}(t))
\ge f_t(\Delta_{i-1}(t))$. Therefore,
\begin{equation}
\label{ieq:psi_1}
\begin{aligned}
\Psi_i(t)-\Psi_{i-1}(t)
\mkern-3mu\ge\mkern-3mu 
-u_i\delta_i(t)+(u_i-u_{i-1})f_t(\Delta_{i-1}(t)),
\end{aligned}
\end{equation}
and
\begin{equation}
\label{ieq:psi_2}
\begin{aligned}
\Psi_{i-2}(t)-\Psi_{i-1}(t)
\ge-(u_{i-1}-u_{i-2}) f_t(\Delta_{i-1}(t)).
\end{aligned}
\end{equation}
Inequality \eqref{ieq:psi_2} also holds for $i=2$ because $u_0=\Psi_0(t)=0$. Since $u_i=i\wedge K$ is discretely concave and $f_t\le0$, we obtain from \eqref{ieq:psi_1} and \eqref{ieq:psi_2}
\begin{equation}
\label{eq:psi_second_lower}
\begin{aligned}
\Psi_i(t)-2\Psi_{i-1}(t)+\Psi_{i-2}(t)
&\ge-u_i\delta_i(t),\qquad i=2,\ldots,N.
\end{aligned}
\end{equation}

Combining \eqref{eq:delta_gap_dynamics} and \eqref{eq:psi_second_lower} yields
\begin{equation}
\label{eq:delta_recursive_differential}
\begin{aligned}
-\dot\delta_i(t)+\bigl(\lambda(t)+u_i\bigr)\delta_i(t)
&\ge\lambda(t)\delta_{i+1}(t),\qquad i=2,\ldots,N-1,
\end{aligned}
\end{equation}
and
\begin{equation}
\label{eq:delta_N_differential}
\begin{aligned}
-\dot\delta_N(t)+\bigl(\lambda(t)+u_N\bigr)\delta_N(t)
\ge\lambda(t)\bigl(\pi-\Delta_N(t)\bigr).
\end{aligned}
\end{equation}
By \eqref{eq:threshold_delta_N_upper_bound},
\begin{equation}
\pi-\Delta_N(t)\ge\pi e^{-\Lambda(t,T)}.
\end{equation}
Using $\delta_N(T)=0$ and applying the integrating-factor formula to \eqref{eq:delta_N_differential}, we obtain
\begin{equation}
\begin{aligned}
\delta_N(t) &\ge\int_t^T e^{-\Lambda(t,s)-u_N(s-t)}\lambda(s) \pi e^{-\Lambda(s,T)}\,ds=\underline{\widehat\delta}_N(t).
\label{eq:delta_recursive_integralN}
\end{aligned}
\end{equation}
Similarly, \eqref{eq:delta_recursive_differential} implies
\begin{equation}
\label{eq:delta_recursive_integral}
\begin{aligned}
\delta_i(t)
\ge\int_t^T e^{-\Lambda(t,s)-u_i(s-t)}
\lambda(s)\delta_{i+1}(s)\,ds,\qquad i=2,\ldots,N-1.
\end{aligned}
\end{equation}
Starting with the bound for $\delta_N$ and applying \eqref{eq:delta_recursive_integral} recursively gives
\begin{equation}
\label{eq:delta_recursive_domination}
\delta_i(t) \ge \underline{\widehat\delta}_i(t), \qquad i=2,\ldots,N,
\end{equation}
and
\begin{equation}
\label{eq:underline_delta_nested}
\begin{aligned}
\underline{\widehat\delta}_i(t)
=\pi e^{-\Lambda(t,T)}
\int_{t\le s_i\le\cdots\le s_N\le T}
e^{-\sum_{m=i}^N u_m(s_m-s_{m-1})}\prod_{m=i}^N\lambda(s_m)\,ds_N\cdots ds_i,
\end{aligned}
\end{equation}
where $s_{i-1}:=t$. These inequalities hold on each continuity interval of $e$ and $\lambda$. Since the marginal costs are continuous at their finitely many discontinuity points, the intervalwise inequalities combine to give \eqref{eq:delta_recursive_domination} on all of $[0,T]$.
Since $u_m\le K$,
\begin{equation}
\begin{aligned}
\sum_{m=i}^N u_m(s_m-s_{m-1})
&\le K(s_N-t)\le K(T-t).
\end{aligned}
\end{equation}
Moreover, the simplex integral identity gives
\begin{equation}
\begin{aligned}
&\int_{t\le s_i\le\cdots\le s_N\le T}
\prod_{m=i}^N\lambda(s_m)\,ds_N\cdots ds_i \mkern-3mu
=\frac{\Lambda(t,T)^{N-i+1}}{(N-i+1)!}.
\end{aligned}
\end{equation}
It follows from \eqref{eq:underline_delta_nested} that
\begin{equation}
\label{eq:underline_delta_individual}
\begin{aligned}
\underline{\widehat\delta}_i(t)\ge\pi e^{-\Lambda(t,T)}e^{-K(T-t)}\frac{\Lambda(t,T)^{N-i+1}}{(N-i+1)!}=\underline\delta_i(t).
\end{aligned}
\end{equation}

As $i$ ranges from $2$ to $N$, $n:=N-i+1$ ranges from $1$ to $N-1$. The ratios of consecutive terms in the sequence $\{\Lambda(t,T)^n/n!\}_{n=1}^{N-1}$ are $\frac{\Lambda(t,T)^{n+1}/(n+1)!}{\Lambda(t,T)^n/n!}=\frac{\Lambda(t,T)}{n+1}$, which are non-increasing in $n$. Hence, the sequence first increases and then decreases, so its minimum is attained at one of the endpoints. Therefore,
\begin{equation}
\begin{aligned}
\min_{1\le n\le N-1}\frac{\Lambda(t,T)^n}{n!}
=\min\left\{\Lambda(t,T),
\frac{\Lambda(t,T)^{N-1}}{(N-1)!}\right\}.
\end{aligned}
\end{equation}
Combining this identity with \eqref{eq:delta_recursive_domination} and \eqref{eq:underline_delta_individual} proves \eqref{deltalb00}.
\end{proof}

\bsk

\section{Proof of Proposition \ref{pro:capacity_marginal_cost}}
\label{app:capacity_control_interaction}

\begin{proof}
For a system with capacities $n$ and $k$, write
\begin{equation}
\label{eq:capacity_psi}
\begin{aligned}
u_i^k&:=i\wedge k,\qquad \Psi_i^{n,k}(t):=u_i^k\min\{0,e(t)-\Delta_i^{n,k}(t)\},
\mkern10mu i=1,\ldots,n,
\end{aligned}
\end{equation}
with $\Psi_0^{n,k}(t):=0$.

We first fix $K$ and compare systems with $N$ and $N+1$ units. Define
\begin{equation}
\label{eq:capacity_unit_difference}
\begin{aligned}
\delta_i^{N,K}(t)
&:=\Delta_i^{N,K}(t)-\Delta_i^{N+1,K}(t),\qquad i=1,\ldots,N.
\end{aligned}
\end{equation}
Suppressing the argument $t$, subtracting the two marginal-cost systems gives
\begin{equation}
\label{eq:capacity_unit_difference_dynamics}
\begin{aligned}
-\dot\delta_i^{N,K}&=\lambda(t)\bigl(\delta_{i+1}^{N,K}-\delta_i^{N,K}\bigr)+\bigl(\Psi_i^{N,K}-\Psi_i^{N+1,K}\bigr)
-\bigl(\Psi_{i-1}^{N,K}-\Psi_{i-1}^{N+1,K}\bigr),\quad i=1,\ldots,N-1,\\[1mm]
-\dot\delta_N^{N,K}&=\lambda(t)\bigl(\pi-\Delta_{N+1}^{N+1,K}
-\delta_N^{N,K}\bigr)+\bigl(\Psi_N^{N,K}-\Psi_N^{N+1,K}\bigr)
-\bigl(\Psi_{N-1}^{N,K}-\Psi_{N-1}^{N+1,K}\bigr).
\end{aligned}
\end{equation}
The terminal conditions imply $\delta_i^{N,K}(T)=0$ for every $i$.

Consider the reversed-time dynamics at a boundary point where
$\delta_i^{N,K}=0$ and $\delta_j^{N,K}\ge0$ for every $j$. By
\eqref{eq:capacity_unit_difference},
$\Delta_i^{N,K}=\Delta_i^{N+1,K}$, and hence $\Psi_i^{N,K}=\Psi_i^{N+1,K}$. Moreover, $\delta_{i-1}^{N,K}\ge0$ and the map
$z\mapsto\min\{0,e(t)-z\}$ is non-increasing, so $\Psi_{i-1}^{N,K}-\Psi_{i-1}^{N+1,K}\le0$. For $i=1$, the latter difference is zero by the convention
$\Psi_0^{n,k}=0$. Therefore, if $i<N$, the right-hand side of
\eqref{eq:capacity_unit_difference_dynamics} is at least
$\lambda(t)\delta_{i+1}^{N,K}\ge0$.
If $i=N$, it is at least $\lambda(t)\bigl(\pi-\Delta_{N+1}^{N+1,K}(t)\bigr)\ge0$, where the last inequality follows from Lemma \ref{lem:monotonicity}.

Thus, the tangent condition holds on every boundary face of the nonnegative orthant. Since the vector field is Lipschitz, Nagumo's theorem, applied successively on the continuity intervals of $e$ and $\lambda$, gives
\begin{equation}
\label{eq:capacity_unit_comparison}
\begin{aligned}
\delta_i^{N,K}(t)&\ge0,\\
\Delta_i^{N+1,K}(t)&\le\Delta_i^{N,K}(t),
\qquad i=1,\ldots,N.
\end{aligned}
\end{equation}
The marginal costs are continuous at the finitely many discontinuity points, so \eqref{eq:capacity_unit_comparison} holds on all of $[0,T]$.

We next fix $N$ and compare systems with $K$ and $K+1$ servers. Define
\begin{equation}
\label{eq:capacity_server_difference}
\begin{aligned}
\widehat\delta_i^{N,K}(t)
&:=\Delta_i^{N,K}(t)-\Delta_i^{N,K+1}(t),i=1,\ldots,N.
\end{aligned}
\end{equation}
Subtracting the corresponding marginal-cost systems gives
\begin{equation}
\label{eq:capacity_server_difference_dynamics}
\begin{aligned}
&-\dot{\widehat\delta}_i^{N,K}
=\lambda(t)\bigl(\widehat\delta_{i+1}^{N,K}
-\widehat\delta_i^{N,K}\bigr)+\bigl(\Psi_i^{N,K}-\Psi_i^{N,K+1}\bigr)-\bigl(\Psi_{i-1}^{N,K}-\Psi_{i-1}^{N,K+1}\bigr),\quad i=1,\ldots,N-1,\\[1mm]
&-\dot{\widehat\delta}_N^{N,K}
=-\lambda(t)\widehat\delta_N^{N,K}+\bigl(\Psi_N^{N,K}-\Psi_N^{N,K+1}\bigr)-\bigl(\Psi_{N-1}^{N,K}-\Psi_{N-1}^{N,K+1}\bigr).
\end{aligned}
\end{equation}
Again, $\widehat\delta_i^{N,K}(T)=0$ for every $i$.

Consider a boundary point where
$\widehat\delta_i^{N,K}=0$ and
$\widehat\delta_j^{N,K}\ge0$ for every $j$. It suffices to show that $\Psi_i^{N,K}-\Psi_i^{N,K+1}\ge\Psi_{i-1}^{N,K}-\Psi_{i-1}^{N,K+1}$. Write $f_t(z):=\min\{0,e(t)-z\}$. If $i\le K$, the feasible maximum actions coincide under the two server capacities, so
\begin{equation}
\begin{aligned}
\Psi_i^{N,K}-\Psi_i^{N,K+1}&=0,\\
\Psi_{i-1}^{N,K}-\Psi_{i-1}^{N,K+1}&\le0.
\end{aligned}
\end{equation}
If $i=K+1$, then
\begin{equation}
\begin{aligned}
\Psi_i^{N,K}-\Psi_i^{N,K+1}
&=-f_t\bigl(\Delta_i^{N,K}\bigr)\ge0,\\
\Psi_{i-1}^{N,K}-\Psi_{i-1}^{N,K+1}
&\le0.
\end{aligned}
\end{equation}

Finally, suppose $i>K+1$. By Lemma \ref{lem:monotonicity} and
$\widehat\delta_{i-1}^{N,K}\ge0$, we have
\begin{equation*}
\Delta_{i-1}^{N,K+1}
\le\Delta_{i-1}^{N,K}
\le\Delta_i^{N,K}
=\Delta_i^{N,K+1}.
\end{equation*}
Since $f_t$ is non-increasing and nonpositive,
\begin{equation}
\begin{aligned}
\Psi_{i-1}^{N,K}-\Psi_{i-1}^{N,K+1}&=Kf_t\bigl(\Delta_{i-1}^{N,K}\bigr)-(K+1)f_t\bigl(\Delta_{i-1}^{N,K+1}\bigr)\\
&\le-f_t\bigl(\Delta_{i-1}^{N,K}\bigr)\le-f_t\bigl(\Delta_i^{N,K}\bigr)=\Psi_i^{N,K}-\Psi_i^{N,K+1}.
\end{aligned}
\end{equation}
Hence, the reversed-time derivative in
\eqref{eq:capacity_server_difference_dynamics} is nonnegative on every boundary face. Applying Nagumo's theorem successively on the continuity intervals and using continuity at the discontinuity points gives
\begin{equation}
\label{eq:capacity_server_comparison}
\begin{aligned}
\widehat\delta_i^{N,K}(t)&\ge0,\\
\Delta_i^{N,K+1}(t)&\le\Delta_i^{N,K}(t),
\qquad i=1,\ldots,N.
\end{aligned}
\end{equation}
\end{proof}


%

\section{Proofs for \S \ref{sec:cyclic_model}}
\label{app:cross_cycle}

\begin{proof}[Proof of Theorem \ref{theo:cross_cycle}.]
For $m=1,\ldots,M$, define $\Delta_i^m(s):=\Delta_i((m-1)L+s)$, for $s\in[0,L]$ and $i=1,\ldots,N$. Under \eqref{eq:cyclic_parameters}, these functions satisfy the same marginal-cost HJB equations within every cycle.

We first show that an ordering at the end of a cycle is preserved backward throughout the cycle. Consider two functions $x=(x_1,\ldots,x_N)$ and $y=(y_1,\ldots,y_N)$ satisfying the same marginal-cost HJB equations on $[0,L]$, and suppose that $x_i(L)\geq y_i(L)$ for every $i$. Let $u_i:=i\wedge K$, $u_0:=0$, and define $f_i(s,z):=u_i\min\{0,\widetilde e(s)-z\}$ and $f_0:=0$. Their HJB equations can be written as
\begin{equation}
\label{eq:cross_cycle_hjb}
\begin{aligned}
-\dot x_i(s)
=\widetilde\lambda(s)\bigl(\pi\mathbf 1_{\{i=N\}}+x_{i+1}(s)-x_i(s)\bigr)+f_i(s,x_i(s))-f_{i-1}(s,x_{i-1}(s)),
\end{aligned}
\end{equation}
and analogously for $y$.

Reverse time by defining $R_i(\tau):=x_i(L-\tau)-y_i(L-\tau)$. Consider a boundary point of the nonnegative orthant at which $R_i(\tau)=0$ and $R_j(\tau)\geq0$ for every $j$. Setting $s=L-\tau$ and subtracting the two equations in \eqref{eq:cross_cycle_hjb} gives
\begin{equation}
\label{eq:cross_cycle_boundary}
\begin{aligned}
\dot R_i(\tau)
=\widetilde\lambda(s)R_{i+1}(\tau)-\Bigl[f_{i-1}(s,x_{i-1}(s))-f_{i-1}(s,y_{i-1}(s))\Bigr]\ge0.
\end{aligned}
\end{equation}
Indeed, $R_{i+1}(\tau)\geq0$, while $f_{i-1}(s,\cdot)$ is non-increasing, so $R_{i-1}(\tau)\geq0$ implies $f_{i-1}(s,x_{i-1}(s))\leq f_{i-1}(s,y_{i-1}(s))$. The vector field is Lipschitz in the marginal costs on every continuity interval of $\widetilde\lambda$ and $\widetilde e$. Nagumo's theorem therefore implies that $R_i(\tau)\geq0$ for every $i$ and $\tau\in[0,L]$. Applying this argument successively across the finitely many discontinuity points, at which the marginal costs remain continuous, shows that
\begin{equation}
\label{eq:cross_cycle_comparison}
\begin{aligned}
x_i(L)\geq y_i(L)\quad\text{for all }i\qquad\Longrightarrow\quad
x_i(s)\geq y_i(s) \quad\text{for all }i\text{ and }s\in[0,L].
\end{aligned}
\end{equation}

We now apply \eqref{eq:cross_cycle_comparison} recursively. For the last two cycles, the intercycle matching condition, Lemma~\ref{lem:monotonicity}, and the terminal condition give
\begin{equation}
\Delta_i^{M-1}(L)=\Delta_i^M(0)
\geq0=\Delta_i^M(L).
\end{equation}
Hence, \eqref{eq:cross_cycle_comparison} yields $\Delta_i^{M-1}(s)\geq\Delta_i^M(s)$ for every $i$ and $s\in[0,L]$.

Proceeding backward, suppose that $\Delta_i^m(s)\geq\Delta_i^{m+1}(s)$ for every $i$ and $s\in[0,L]$. The intercycle matching conditions imply
\begin{equation}
\Delta_i^{m-1}(L)=\Delta_i^m(0)
\geq\Delta_i^{m+1}(0)=\Delta_i^m(L).
\end{equation}
Another application of \eqref{eq:cross_cycle_comparison} gives $\Delta_i^{m-1}(s)\geq\Delta_i^m(s)$ throughout the preceding cycle. Repeating this argument proves
\begin{equation}
\begin{aligned}
\Delta_i^m&(s)\geq\Delta_i^{m+1}(s),\qquad m=1,\ldots,M-1,
\quad i=1,\ldots,N,
\quad s\in[0,L],
\end{aligned}
\end{equation}
which is equivalent to \eqref{eq:cross_cycle_marginal_cost}.

Finally, \eqref{eq:cyclic_parameters} gives $e((m-1)L+s)=e(mL+s)=\widetilde e(s)$ for $s\in[0,L)$. Therefore, $\Delta_i(mL+s)>e(mL+s)$ implies $\Delta_i((m-1)L+s)>e((m-1)L+s)$, and hence
\begin{equation}
\mathcal I_{(m-1)L+s}\supseteq\mathcal I_{mL+s}.
\end{equation}
Taking the smallest elements under the convention for an empty activation set yields $i_{(m-1)L+s}^{\min}\leq i_{mL+s}^{\min}$.
\end{proof}

\hsk

\begin{proof}[Proof of Proposition~\ref{prop:peak_offpeak}.]
By \eqref{eq:threshold_delta_N_upper_bound},
\begin{equation}
\Delta_N(t)\leq
\pi\bigl(1-e^{-\Lambda(t,T)}\bigr).
\end{equation}
The right-hand side is decreasing in $t$. Combining this bound with
the first condition in \eqref{eq:peak_offpeak_condition} and
Lemma~\ref{lem:monotonicity} gives
\begin{equation}
\label{eq:app-peak-shutdown-bound}
\begin{aligned}
\Delta_i(t)
&\leq\Delta_N(t)\leq\pi\left(1-e^{-\Lambda(t,T)}\right)
\leq e_{\mathrm H},\qquad i=1,\ldots,N,
\quad t\in\{(m-1)L,mL\}.
\end{aligned}
\end{equation}
At $t=(m-1)L$ and $t=mL$, respectively,
\eqref{eq:app-peak-shutdown-bound} implies that the activation threshold is $N+1$. Because $e(t)=e_{\mathrm H}$ on $H_m$ and $H_{m+1}$, Corollary~\ref{cor:increasing_operating_cost} implies that the threshold remains $N+1$ throughout both intervals. Hence, Proposition~\ref{pro:optimal_control} gives $u^*(t,i)=0$ for every $i=0,\ldots,N$ and $t\in H_m\cup H_{m+1}$.

It remains to establish full activation on $L_m$. Let $A$ be the number of arrivals during $H_{m+1}$. Because arrivals form a nonhomogeneous Poisson process,
\begin{equation*}
A\sim\operatorname{Poisson}(\Lambda_{\mathrm H}),
\qquad
\Lambda_{\mathrm H}=\Lambda(mL,(m+\rho)L).
\end{equation*}

Fix $i\in\{1,\ldots,N\}$, and couple two systems starting at time $mL$ in states $i$ and $i-1$ under the same arrival process. Because replenishment is inactive throughout $H_{m+1}$, the difference between their numbers of lost demands is
\begin{equation}
\label{eq:app-peak-lost-demand-difference}
\begin{aligned}
\left[A-(N-i)\right]^+
-\left[A-(N-i+1)\right]^+=\mathbf 1\{A\geq N-i+1\}.
\end{aligned}
\end{equation}
Conditional on $A=k\leq N-i$, the two systems enter time
$(m+\rho)L$ in states $i+k$ and $i+k-1$, respectively. If $A\geq N-i+1$, both systems enter that time in state $N$. Consequently, \eqref{eq:app-peak-lost-demand-difference} yields
\begin{equation}
\label{eq:app-peak-marginal-decomposition}
\begin{aligned}
\Delta_i(mL)=\pi\Pr\{A\geq N-i+1\}+\sum_{k=0}^{N-i}\Pr\{A=k\}
\Delta_{i+k}\bigl((m+\rho)L\bigr).
\end{aligned}
\end{equation}
The marginal continuation costs are nonnegative. Setting $i=1$ in
\eqref{eq:app-peak-marginal-decomposition} therefore gives
\begin{equation}
\label{eq:app-peak-smallest-marginal-bound}
\begin{aligned}
\Delta_1(mL)
\geq \pi\Pr\{A\geq N\}=\pi\left[1-e^{-\Lambda_{\mathrm H}}
\sum_{k=0}^{N-1}
\frac{\Lambda_{\mathrm H}^{k}}{k!}\right].
\end{aligned}
\end{equation}
The second condition in \eqref{eq:peak_offpeak_condition}, together with
\eqref{eq:app-peak-smallest-marginal-bound}, implies $e_{\mathrm L}<\Delta_1(mL)$. By continuity, $\Delta_1(mL^-)=\Delta_1(mL)>e_{\mathrm L}$, and Lemma~\ref{lem:monotonicity} then gives $\Delta_i(mL^-)>e_{\mathrm L}$ for every $i=1,\ldots,N$. Hence, $i_{mL^-}^{\min}=1$.

Finally, Corollary~\ref{cor:increasing_operating_cost} implies that the threshold is increasing on $L_m$. Thus,
\begin{equation*}
1\leq i_t^{\min}\leq i_{mL^-}^{\min}=1, \qquad t\in L_m,
\end{equation*}
so $i_t^{\min}=1$ throughout $L_m$. Proposition~\ref{pro:optimal_control} therefore gives $u^*(t,i)=\min\{K,i\}$ for every $i=0,\ldots,N$ and $t\in L_m$. Combining this result with shutdown on $H_m$ proves the proposition.
\end{proof}

\end{document}